\documentclass[a4paper,11pt]{article}

\usepackage{amsmath,amsthm,amssymb,mathrsfs,latexsym,amsfonts}
\usepackage{graphicx,psfrag,epsfig}
\usepackage{bbm}
\usepackage{a4wide}
\usepackage[english]{babel}
\usepackage[latin1]{inputenc}
\usepackage{authblk}

\numberwithin{equation}{section}

\newtheorem{theorem}{Theorem}[section]
\newtheorem{lemma}[theorem]{Lemma}
\newtheorem{proposition}[theorem]{Proposition}

\newtheorem{corollary}[theorem]{Corollary}
\newtheorem{definition}{Definition\rm}

\newcounter{paraga}[section]

\newcommand{\etoile}{^*}
\newcommand{\N}{\mathbb{N}}
\newcommand{\Z}{\mathbb{Z}}
\newcommand{\Q}{\mathbb{Q}}
\newcommand{\R}{\mathbb{R}}
\newcommand{\C}{\mathbb{C}}

\newcommand{\eps}{\varepsilon}
\newcommand{\al}{\alpha}

\newcommand{\Span}{{\rm Span}}

\newcommand{\lam}{\lambda}

\newcommand{\Lam}{\Lambda}
\newcommand{\calC}{{\mathcal{C}}}
\newcommand{\lamu}{\lambda_1(\calC, \Lam)}
\newcommand{\lamk}{\lambda_k(\calC, \Lam)}
\newcommand{\lamd}{\lambda_d(\calC, \Lam)}
\newcommand{\lamkmuc}{\lambda_k(\mu \calC, \Lam)}
\newcommand{\lamkmulam}{\lambda_k( \calC, \mu\Lam)}

\newcommand{\Lamet}{\Lambda\etoile}
\newcommand{\calCet}{{\mathcal{C}}\etoile}

\newcommand{\lamdpumket}{\lambda_{d+1-k}(\calCet, \Lamet)}

\newcommand{\Lampr}{\Lambda '}
\newcommand{\calCpr}{{\mathcal{C}} ' }
\newcommand{\lamkpr}{\lambda_k(\calCpr, \Lampr)}

\newcommand{\vol}{{\rm vol}}

\newcommand{\und}{\{1,\ldots,d\}}

\newcommand{\pphk}{\Psi_k}

\newcommand{\etoi}{'}
\newcommand{\Psiet}{\Psi\etoi}
\newcommand{\ppsk}{\Psiet_k}
\newcommand{\ppsdmk}{\Psiet_{d+1-k}}

\newcommand{\cc}[1]{\calC(#1) }
\newcommand{\ccpr}[1]{\calC' (#1) }
\newcommand{\ccet}[1]{\calC\etoile(#1) }

\begin{document}

\def\MP{\,{<\hspace{-.5em}\cdot}\,}
\def\SP{\,{>\hspace{-.3em}\cdot}\,}
\def\PM{\,{\cdot\hspace{-.3em}<}\,}
\def\PS{\,{\cdot\hspace{-.3em}>}\,}
\def\EP{\,{=\hspace{-.2em}\cdot}\,}
\def\PP{\,{+\hspace{-.1em}\cdot}\,}
\def\PE{\,{\cdot\hspace{-.2em}=}\,}
\def\N{\mathbb N}
\def\C{\mathbb C}
\def\Q{\mathbb Q}
\def\R{\mathbb R}
\def\T{\mathbb T}
\def\A{\mathbb A}
\def\Z{\mathbb Z}
\def\demi{\frac{1}{2}}

\begin{titlepage}
\author{Abed Bounemoura\footnote{Institute for Advanced Study, Einstein Drive, Princeton, NJ 08540, USA} {} and St\'ephane Fischler\footnote{Laboratoire de math\'ematiques d'Orsay, Univ Paris Sud, 91405 Orsay Cedex, France}}
\title{\LARGE{\textbf{A Diophantine duality applied to the KAM and Nekhoroshev theorems}}}
\end{titlepage}

\newcommand{\pb}[1]{{\bf #1}}

\maketitle

\begin{abstract}
In this paper, we use geometry of numbers to relate two dual Diophantine problems. This allows us to focus on simultaneous approximations rather than small linear forms. As a consequence, we develop a new approach to the perturbation theory for quasi-periodic solutions dealing only with periodic approximations and avoiding classical small divisors estimates. We obtain two results of stability, in the spirit of the KAM and Nekhoroshev theorems, in the model case of a perturbation of a constant vector field on the $n$-dimensional torus. Our first result, which is a Nekhoroshev type theorem, is the construction of a ``partial" normal form, that is a normal form with a small remainder whose size depends on the Diophantine properties of the vector. Then, assuming our vector satisfies the Bruno-Rüssmann condition, we construct an ``inverted" normal form, recovering the classical KAM theorem of Kolmogorov, Arnold and Moser for constant vector fields on torus. 
\end{abstract}
 
\section{Introduction and results}

Let $\alpha$ be a non-zero vector in $\R^n$. A quasi-periodic solution with frequency $\alpha$ is simply an integral curve of the constant vector field $X_\alpha=\alpha$ on the $n$-dimensional torus $\T^n=\R^n/\Z^n$. Constant vector fields on a torus have only trivial dynamical properties, but this is not the case for their analytic perturbations, which are of great interest both in theoretical and practical aspects in celestial mechanics and Hamiltonian systems.

First let us recall the dynamical properties of a constant vector field $X_\alpha$, which depends on the ``Diophantine" properties of $\alpha$ (note that these properties depend on $\alpha$ only through its equivalence class under the action of $\mathrm{PGL}(n,\Z)$). If the components of the vector $\alpha$ are independent over the field of rational numbers $\Q$, then $X_\alpha$ is minimal, that is every orbit is dense on $\T^n$, and moreover $X_\alpha$ is uniquely ergodic, that is every orbit is equidistributed with respect to the Haar measure. If the components of the vector $\alpha$ are dependent over $\Q$, and if $m$ is the number of independent rational relations, then there is an invariant foliation on $\T^n$, whose leaves are diffeomorphic to $\T^{n-m}$ and whose leaf space is diffeomorphic to $\T^m$, and the restriction of $X_\alpha$ to each leaf is minimal and uniquely ergodic. Here we will not always make the assumption that $m=0$, and we will write $d=n-m$. 

Adding a small analytic perturbation $P$ to our vector field $X_\alpha$, a simple description of the orbits of $X_\alpha+P$ is in general no longer available, and since Poincaré it is customary to try to conjugate the perturbed vector field to a ``simpler" one, which is usually called a normal form. In our context, the best one can hope for is to conjugate the perturbed vector field to a vector field of the form $X_\alpha+N$, where the vector field $N$ commutes with $X_\alpha$. To construct such a conjugacy to a normal form, one usually has to integrate along the integral curves of $X_\alpha$ and difficulties arise as these curves are in general non-compact: more precisely, values of linear forms $k\cdot\alpha$, $k\in\Z^n$, which might be very small for integers $k$ with large norms, appear in the denominator of the function generating the conjugacy. These are usually called small divisors, and they show that more quantitative information on the values of $|k\cdot\alpha|$, $k\in\Z^n$, that is on the Diophantine properties of $\alpha$, is needed in order to carry on a perturbation theory. By Dirichlet's box principle, when $m=0$ and hence $d=n$, for a given $Q\geq 1$, one can always find a non-zero integer $k$ with $|k|\leq Q$ for which $|k\cdot\alpha|^{-1}\geq |\alpha|^{-1}Q^{n-1}$. In the general case where $d$ is smaller than $n$, adding further multiplicative constants depending on $\alpha$, $Q^{n-1}$ can be replaced by $Q^{d-1}$. This gives a first way of quantifying these Diophantine properties, by introducing a real-valued function depending on $\alpha$, that we shall call here $\Psi'=\Psi'_{\alpha}$, which has to be non-decreasing, piecewise constant, unbounded and grows at least as $Q^{d-1}$ (see \S\ref{ssec1} for a precise definition; this function is denoted by $\Psi'_1$ there and is part of a family of functions $\Psi'_k$ for $1 \leq k \leq d$). Assuming that this function grows at most as a power in $Q$ (which was a condition introduced by Siegel in a related problem, and is known as a Diophantine condition), this problem of small divisors was solved by Kolmogorov (\cite{Kol53}, \cite{Kol54}), who introduced a new fast-converging iteration method similar to a Newton iteration method.

Yet there is one special situation in which these difficulties do not appear, namely in the case $m=n-1$, hence $d=1$, where the integral curves are all periodic, and hence compact. For such a vector, that we shall denote here $\omega$ to distinguish with the general $\alpha$, this means that there exists a real number $T>0$ such that $T\omega\in\Z^n$, and assuming that $T$ is minimal with this property, $T$ is the minimal common period of all the integral curves. By extension, such a vector $\omega$ is called periodic, and it is characterized by the fact the $(n-1)$-vector defined by the ratios of its components (or equivalently, its image in the projective space) is rational. The perturbation theory in this case is much simpler, it was already known to Poincaré and a classical Picard iteration method is here sufficient. By Dirichlet's box principle, it is always possible to approximate a non-zero vector $\alpha$ by a periodic vector $\omega$ as follows: when $m=0$ and hence $d=n$, for a given $Q\geq 1$, one can always find a periodic vector $\omega$, with period $T$, such that $|T\alpha-T\omega|\leq Q^{-1}$ and $|\alpha|^{-1}\leq T \leq |\alpha|^{-1}Q^{n-1}$. As before, in the general case where $d$ is smaller than $n$, one obtains $Q^{d-1}$ instead of $Q^{n-1}$ but with other constants depending on $\alpha$. However, such a periodic approximation does not give us enough information in general.  

It is in fact possible to find not just one but $d$ independent periodic approximations in the following sense: for any $Q\geq 1$, there exists independent periodic vectors $\omega_1,\dots,\omega_d$, with periods $T_1,\dots,T_d$, such that $|T_j\alpha-T_j\omega_j|\leq Q^{-1}$ for $1\leq j \leq d$. This last inequality will be expressed by saying that each $\omega_j$ is a $Q$-approximation of $\alpha$. Of course, in general the upper bound on the periods $T_j$, which is at least of order $Q^{d-1}$, can now be much larger, but we can introduce yet another real-valued function, that we call $\Psi=\Psi_\alpha$, which also has to be non-decreasing, piecewise constant, unbounded and grows at least as $Q^{d-1}$ and which gives the growth of the upper bound on these periods in terms of $Q$ (once again, see \S\ref{ssec1} for a precise definition where this function is denoted by $\Psi_d$ and is also part of a family of functions $\Psi_k$ for $1 \leq k \leq d$, see also Proposition \ref{corutile} in \S \ref{subseccordio} for the result). This is another characterization of the Diophantine properties of $\alpha$, and our main Diophantine result, Theorem~\ref{thdio}, will imply (as a particular case) that this characterization is equivalent to the previous one, in the sense that the functions $\Psi$ and $\Psi'$ are the same up to constants depending on $d$ and $\alpha$ (Theorem~\ref{thdio} states that the functions $\Psi_k$ and $\Psi'_{d+1-k}$ are the same up to constants depending on $d$ and $\alpha$, for any $1\leq k \leq d$). This will be proved in \S\ref{sec2}, by using arguments from the geometry of numbers: more precisely, our result will be a consequence of the relation between the successive minima of a convex body with respect to a given lattice and the successive minima of the polar convex body with respect to the polar lattice.

Then, in \S\ref{sec3}, we shall develop a new approach to the perturbation theory for quasi-periodic solutions based uniquely on periodic approximations. Roughly speaking, from an analytical point of view, we will reduce the general quasi-periodic case $1\leq d \leq n$ to the periodic case $d=1$. One could probably say that perturbation theory involves some arithmetics, some analysis and some geometry. Since we have nothing new to offer as far as the geometry is concerned, we shall restrict to a situation where it simply does not enter into the picture, and this is the case if the frequency of the solution is fixed. This is why we will consider perturbation of constant vector fields on the torus (we could have also considered perturbation of linear integrable Hamiltonian systems, this is quite similar). In this context, we will apply our method and prove two results in the spirit of the KAM and Nekhoroshev theorems. Our first result, Theorem~\ref{thmfini}, which is valid for any $1\leq d \leq n$ and any $\alpha \in \R^n\setminus\{0\}$, is the existence of an analytic conjugacy to a ``partial" normal form, that is we will construct an analytic conjugacy between $X_\alpha+P$ and a vector field $X_\alpha+N+R$, where $X_\alpha+N$ is a normal form and $R$ a ``small" remainder. The smallness of this remainder depends precisely on the Diophantine properties of $\alpha$, and if the latter satisfies a classical Diophantine condition, the remainder is exponentially small (up to an exponent) with respect to the inverse of the size of the perturbation. This statement is analogous to the Nekhoroshev theorem for perturbation of linear integrable Hamiltonian systems. Our second result, Theorem~\ref{thminfini}, which is valid only for $d=n$ and for $\alpha \in \R^n\setminus\{0\}$ satisfying the Bruno-Rüssmann condition, is the existence of an analytic conjugacy between a modified perturbed vector field $X_\alpha+X_\beta+P$, where $X_\beta=\beta$ is another constant vector field that depends on $P$, and the unperturbed vector field $X_\alpha$. Note that for $d=n$, constant vector fields are exactly vector fields $N$ for which $[N,X_\alpha]=0$, hence the theorem states the existence of a vector field $N$ which commutes with $X_\alpha$ such that $X_\alpha+N+P$ is analytically conjugated to $X_\alpha$, and therefore this can be called an ``inverted" normal form. This statement is exactly the classical KAM theorem for vector fields on the torus of Kolmogorov, Arnold and Moser.

To conclude, let us point out that the idea of introducing periodic approximations in perturbation theory is due to Lochak. In \cite{Loc92}, using only Dirichlet's box principle, Lochak gave drastic improvements, both from a qualitative and a quantitative point of view, on Nekhoroshev's stability estimates for a perturbation of a convex integrable Hamiltonian system. Lochak's argument strongly relies on convexity, and this assumption enables him to use just one periodic approximation to derive an essentially ``optimal" result. However, in the general case, one periodic approximation seems to be not enough. The idea has been then taken up by Niederman in \cite{Nie07}, and further improved in \cite{BN09}, where linearly independent periodic approximations were used to prove Nekhoroshev's estimates for a much larger class of integrable Hamiltonian systems. However, the arguments there were based only on successive applications of Dirichlet's box principle, and unlike \cite{Loc92}, the results obtained were far from being ``optimal" (this also stems from the fact that Nekhoroshev's estimates are much more complicated to derive without the convexity assumption). At last, we should also point out that using a multi-dimensional continued fraction algorithm due to Lagarias, Khanin, Lopes Dias and Marklof gave a proof of the KAM theorem with techniques closer to renormalization theory (see \cite{KDM06}, \cite{KDM07}). Even though their arguments deal with small divisors yet in a different way, they were still based on Fourier expansions and elimination of Fourier modes, which is quite different from the techniques we will use in this paper, since Fourier analysis will not be involved at all.


\section{A Diophantine duality}\label{sec2}

\newcommand{\etapr}{\eta'}
\newcommand{\Ek}{{\mathcal E}_k}
\newcommand{\Fk}{{\mathcal F}_k}
\newcommand{\Eun}{{\mathcal E}_1}
\newcommand{\Ede}{{\mathcal E}_2}
\newcommand{\es}{\underline e}
\newcommand{\el}{\es_\ell}
\newcommand{\eprl}{\es'_\ell}
\newcommand{\eprlpu}{\es'_{\ell+1}}
\newcommand{\eprs}{\es'}
\newcommand{\eun}{\es_1}
\newcommand{\eprun}{\es'_1}
\newcommand{\elpu}{\es_{\ell+1}}
\newcommand{\Xl}{X_\ell}
\newcommand{\Xprl}{X'_\ell}
\newcommand{\unk}{\{1,\ldots,k\}}
\newcommand{\Cpsipr}{C'_k}

We gather in this section all the Diophantine part of this text. In \S \ref{ssec1} we define the functions $\Psi_k$ and $\Psi'_k$ mentioned in the Introduction, and we state our main result: $\Psi_k$ and $\Psi'_{d+1-k}$ are essentially equal. We don't claim this result to be new, but we did not find it in the literature so we provide a complete proof based on a classical duality result of geometry of numbers. 
In the special case $k=d$ (which is the one used in the rest of this paper), this Diophantine duality is related to inhomogeneous approximation 
(see \cite{Casselsdio}, Chapter V, \S 9), in particular to transference results due to Khintchine (see Theorem VI in Chapter XI, \S 3 of  \cite{Cassels},
Theorem XVII in Chapter V, \S 8 of  \cite{Casselsdio}, and also Lemma 3 of \cite{BLMoscow2005} and Lemma 4.1 of \cite{MiWDEA}, p. 143). We would like to point out the fact that these transference results, between homogeneous and inhomogeneous Diophantine approximations, are distinct from the well-known ``Khintchine's transference principle'' (stated as Theorem 5A, p. 95 of \cite{Schmidt}) which connects the existence of solutions in dual homogeneous Diophantine approximation problems. By the way we believe that this transference principle of Khintchine is in general not good enough for the applications to the KAM and Nekhoroshev theorems (except in the very special case where the vector satisfies a classical Diophantine condition with the best exponent).

In \S \ref{subseccordio} we derive some corollaries and state other useful Diophantine properties (some of them related to Diophantine vectors and the Bruno-Rüssmann condition). In this way, all Diophantine properties used in our application to the perturbation theory for quasi-periodic solutions are stated in \S  \ref{subseccordio}.

We recall in \S \ref{ssec2} the definition and classical properties of the successive minima of a convex body with respect to a lattice, including the main result  from geometry of numbers we rely on (namely \eqref{eqdualCassels}). Using these tools we prove in \S \ref{ssec3} the results stated in  \S\S \ref{ssec1} or  \ref{subseccordio}, except for Proposition \ref{proppsietpsiprime} which is proved in \S \ref{subsecminifam}. At last, we focus in \S \ref{subsecunnb} on the case of one number, in which the continued fraction expansion enables one to compute $\Psi_k$ and $\Psi'_k$ explicitly.

\subsection{The general result} \label{ssec1}

Let $n\geq 2$, and $\al = (\al_0,\ldots,\al_{n-1})\in \R^{n}\setminus\{0\}$. Without loss of generality, we assume that $|\al|=|\al_0|$, where $|\al|=\max_{0\leq i \leq n-1}|\al_i|$. 

Let us say that a vector subspace of $\R^n$ is rational if it possesses a basis of vectors in $\Q^n$ (see for instance \cite{Boualg}, \S 8), and let $F=F_\alpha$ be the smallest rational subspace of $\R^n$ containing $\alpha$. The dimension $d$ of $F$, which is also the rank of the $\Z$-module $\Z^n \cap F$, is called the number of effective frequencies of $\al$. Of course, $F=\R^n$ if and only if $\al_0, \dots, \al_{n-1}$ are linearly independent over $\Q$.

\begin{definition} \label{defphi} For $k \in \{1,\dots,d\}$ and a real number $Q \geq 1 $, we let $\Psi_k(Q)$ denote the infimum of the set of $K > 0$ such that there exist $k$ linearly independent vectors $x\in\Z^n\cap F$ satisfying the inequalities
\begin{equation}\label{eqphi} 
\begin{cases}
| x | \leq K, \\ 
| x_0 \alpha - \alpha_0 x | \leq Q^{-1}. 
\end{cases}
\end{equation}
\end{definition}

For any $Q\geq 1$, $\Psi_k(Q)$ exists because it is  the infimum of a non-empty set, that is for $Q\geq 1$ and $K>0$ sufficiently large, the set of $k$ linearly independent vectors $x\in\Z^n\cap F$ satisfying~\eqref{eqphi} is non-empty; we will prove this in \S \ref{ssec3} (within the proof of Theorem \ref{thdio}), eventhough a direct proof would be possible. This set being also finite, the infimum in the above definition is in fact a minimum. In \S\ref{sec3}, we will use $\pphk$ only when $k=d$, which amounts to finding a basis of $F$ consisting in vectors of $\Z^n$ such that \eqref{eqphi} hold. On the other hand, when $k=1$ the problem is to find a non-zero vector in $\Z^n\cap F$ satisfying these inequalities. We will prove in \S\ref{subseccordio}, Proposition~\ref{corutile}, that for $Q$ sufficiently large, each non-zero integer vector satisfying \eqref{eqphi} gives a periodic vector $\omega$, with a period $T$ bounded (up to a constant) by $K$, which is essentially a $Q$-approximation of $\alpha$ in the sense that $|T\alpha-T\omega|$ is bounded (up to a constant) by $Q^{-1}$. Hence when $k=d$, we will find $d$ linearly independent periodic vectors, with periods essentially bounded by $\Psi_d(Q)$, and which are essentially $Q$-approximations of $\alpha$. Of course, $\Psi_k(Q)$ depends on $\alpha$, but for simplicity, we shall forget this dependence in the notation (except in \S\ref{sec3}, where the function $\Psi_d$ will be denoted by $\Psi_\alpha$). This function $\Psi_k$ is central in our approach; it is non-decreasing,  piecewise constant, and left-continuous (see Proposition \ref{proppsietpsiprime} in \S \ref{subseccordio} below). 

Definition \ref{defphi} can be easily stated in terms of successive minima (see \S \ref{ssec2} below). A variant of this definition would be  to ask for a family of $k$ linearly independent vectors  which can be completed to form a basis of the $\Z$-module $\Z^n\cap F$ (see Proposition \ref{corutile} in \S \ref{subseccordio} below); this would induce only minor changes. 

The classical properties of successive minima (recalled in  \S \ref{ssec2}) will allow us in  \S \ref{ssec3} to prove that $\pphk(Q)$ exists, that it is a positive real number for any $Q\geq 1$, and to relate the function $\pphk$ to the ``dual'' one $\ppsk$ we define now.

\begin{definition} \label{defpsi} For $k \in \{1,\dots,d\}$ let $\Cpsipr $ denote the least integer $Q \geq 1$ for which there exist $k$ linearly independent vectors $x\in\Z^n\cap F$ such that $|x|\leq Q$. For any   real number $Q \geq \Cpsipr $, we let $\Psi'_k(Q)$ denote the supremum of the set of $K > 0$ such that there exist $k$ linearly independent vectors $x\in\Z^n\cap F$ satisfying the inequalities
\begin{equation}\label{eqpsi}
\begin{cases}
| x | \leq Q,  \\ 
| x\cdot\alpha | \leq K^{-1}. 
\end{cases}
\end{equation}
\end{definition}

It is clear that $\Psi'_k(Q)$ exists for any $Q\geq \Cpsipr$ because there are only finitely many non-zero vectors $x\in \Z^n\cap F$ with $|x|\leq Q$, and they satisfy $x\cdot \alpha \neq 0$. From this it is also clear that the supremum in the above definition is in fact a maximum. The constant $\Cpsipr$ will not be very important in this paper, because we are specially interested in $\Psi'_k(Q)$ when $Q$ tends to infinity. Anyway, if $\al_0$, \ldots, $\al_{n-1}$ are linearly independent over $\Q$ then $\Cpsipr=1$ because $F=\R^n$.

The most interesting case is $k=1$, which amounts to  finding a non-zero vector in $\Z^n\cap F$ satisfying \eqref{eqpsi}. Of course, $\Psi'_k(Q)$ depends on $\alpha$, but as before, we shall not take into account this dependence in the notation. 

We will prove  in \S \ref{ssec3} that   the functions $\pphk$ and $\ppsdmk$ are equal up to constants that will be irrelevant in the applications. The precise result is the following (in which we have not tried to optimize the  constants neither the dependence in the lattice $\Z^n \cap F$).

\begin{theorem}\label{thdio} 
Let $Q_0=\max(1, (n+2)\Cpsipr |\al|^{-1})$, and
\[ c_1=(n+1)^{-1}|\al|, \; c_2=(n+2)^{-1}|\al|, \; c_3=(\det \Lam)^2 d! |\al|, \; c_4=2c_3.  \]
If  $Q \geq Q_0$,  then we have the inequalities
$$c_1\ppsdmk (c_2Q)  \leq \pphk(Q) \leq  c_3\ppsdmk (c_4Q).$$
\end{theorem}

As the proof shows, it is possible to replace $c_2$ with any real number less than $c_1$. With this improvement, the first inequality becomes essentially optimal in the case $n=2$ (see 
Proposition  \ref{propexplipsipsipr} in \S \ref{subsecunnb} below, and the remark following it). On the opposite, the constants $c_3$ and $c_4$ can be improved (for instance using refined versions of \eqref{eqdualCassels}), but we did not try to go any further in this direction because the values of these constants are not relevant in the applications we have in mind.

\subsection{Other Diophantine results} \label{subseccordio}

We gather in this section some corollaries of Theorem \ref{thdio}, and other Diophantine results useful in this paper.  
To begin with, let us state some properties of the functions $\Psi_k$ and $\Psi'_k$.

\begin{proposition}\label{proppsietpsiprime}
Let $k\in\und$. The functions $\Psi_k$ and $\Psi'_k$ are non-decreasing, piecewise constant,  and have limit $+\infty$ as $Q\to+\infty$. More precisely, there exist increasing sequences $(Q_\ell)_{\ell\geq 1}$ and $(Q'_\ell)_{\ell\geq 1}$, with limit $+\infty$, such that:
\begin{itemize}
\item $\Psi_k$ is constant on each interval $]Q_\ell, Q_{\ell+1}]$ and on $[1,Q_1]$, with $Q_1\geq 1$;
\item$\Psi'_k$ is constant on each interval $[Q'_\ell, Q'_{\ell+1}[$, with $Q'_1=\Cpsipr$;
\item The values $\Psi_k(Q_\ell)$ and the arguments $Q'_\ell$ are integers.
\end{itemize}
In particular, $\Psi_k$ is  left-continuous and $\Psi'_k$ is  right-continuous.
\end{proposition}

It is not difficult to prove this result directly; however we will deduce it in \S \ref{subsecminifam} from an interpretation of $\Psi_k$ and $\Psi'_k$ in terms of sequences of minimal families, in the spirit of Davenport and Schmidt.

\bigskip

Then, let us recast the definition of the function $\Psi_d$ in terms of approximations of $\alpha$ by periodic vectors.  Recall that a vector $\omega\in\R^n\setminus\{0\}$ is $T$-periodic, $T>0$, if $T\omega\in\Z^n$.

\begin{proposition} \label{corutile}
For any $Q>\max(1,d|\alpha|^{-1})$, there exists $d$ periodic vectors $\omega_1, \dots, \omega_d$, of periods $T_1, \dots, T_d$, such that $T_1\omega_1, \dots, T_d\omega_d$ form a $\Z$-basis of $\Z^n \cap F$ and for $j\in\{1,\dots,d\}$,
\[ |\alpha-\omega_j|\leq d(|\alpha|T_j Q)^{-1}, \quad |\alpha|^{-1} \leq T_j \leq |\alpha|^{-1}d\Psi_d(Q).\]
\end{proposition} 

We will prove this proposition at the end of \S \ref{ssec3} below, after the proof of Theorem~\ref{thdio}.

\bigskip

Finally, we have stated Theorem~\ref{thdio} in a very general setting, since we believe that this generality can be useful in applications. However, in \S\ref{sec3}, we will be  mainly interested in the special case where $k=d$, and Theorem~\ref{thdio} states that the functions $\Psi_d$ and $\Psi_1'$ are equal up to constants. Usually, in the perturbation theory of quasi-periodic solutions, the Diophantine properties of $\alpha$ are quantified by the function $\Psi'_1$: applying Theorem~\ref{thdio}, we will now reformulate these properties in terms of the function $\Psi_d$.  

\begin{definition}\label{defdio}
Given a real number $\tau\geq d-1$, a vector $\alpha$ belongs to $\mathcal{D}_d^\tau$ if and only if there exist two constants $Q'_\alpha \geq C'_1$ and $C'_\alpha >0$ such that for all $Q\geq Q'_\alpha$, $\Psi_1'(Q)\leq C'_\alpha Q^\tau$.
\end{definition}

Vectors in $\mathcal{D}_d^\tau$ are called Diophantine vectors with exponent $\tau$; they are usually defined only when $d=n$. The corollary below is an immediate consequence of Theorem~\ref{thdio}.

\begin{corollary}\label{cordio}
Given  $\tau\geq d-1$, a vector $\alpha$ belongs to $\mathcal{D}_d^\tau$ if and only if there exist two constants $Q_\alpha \geq 1$ and $C_\alpha > 0$ such that for all $Q\geq Q_\alpha$, $\Psi_d(Q)\leq C_\alpha Q^\tau$.
\end{corollary}

Another class of vectors which is well-studied in connection with perturbation of analytic quasi-periodic solutions is the class of vectors satisfying the Bruno-Rüssmann condition.

\begin{definition}\label{defbr}
A vector $\alpha$ belongs to $\mathcal{B}_d$ if and only if there exist a continuous, non-decreasing and unbounded function $\Phi': [1,+\infty[ \rightarrow [1,+\infty[$ such that $\Phi' \geq \Psi'_1$ and 
\[ \int_{1}^{+\infty}Q^{-2}\ln(\Phi'(Q))dQ < \infty. \]
\end{definition}

Vectors in $\mathcal{B}_d$ are called Bruno-Russmann vectors, and as before, they are usually defined only in the case $d=n$. This is not the original definition of Rüssmann (see \cite{Rus01} for instance), but is equivalent to it (for $d=n$): instead of $\Phi'\geq \Psi_1'$, one requires that $|k\cdot\alpha|^{-1}\leq \Phi'(|k|)$ for all $k\in\Z^n\setminus\{0\}$, but in view of the definition of $\Psi_1'$ and the fact that the supremum is reached in this definition, these are equivalent.  

The corollary below is also an immediate consequence of Theorem~\ref{thdio}.

\begin{corollary}\label{corbr}
A vector $\alpha$ belongs to $\mathcal{B}_d$ if and only if there exist a continuous, non-decreasing and unbounded function $\Phi: [1,+\infty[ \rightarrow [1,+\infty[$ such that $\Phi \geq \Psi_d$ and 
\[ \int_{1}^{+\infty}Q^{-2}\ln(\Phi(Q))dQ < \infty. \]
\end{corollary} 

In the appendix, we will show in fact that one can always find a continuous, non-decreasing and unbounded function $\Phi : [1,+\infty[ \rightarrow [1,+\infty[$ such that $\Psi_d(Q) \leq \Phi(Q) \leq \Psi_d(Q+1)$ for $Q\geq 1$: therefore it is equivalent to ask for the integral condition with $\Psi_d$ instead of $\Phi$. Of course, the same comment applies to $\Psi'_1$ in the original definition.

\subsection{Successive minima} \label{ssec2}

In this section we state the definition, and several classical properties, of the successive minima of a convex body with respect to a lattice. Standard references on  this topic include   \cite{Cassels} (Chapter VIII) and \cite{GL}. Classical results will be sufficient for our purposes; they are contained in  the first edition \cite{Lseul} of \cite{GL}. The interested  reader will find in these references much more than what we present here, including proofs, attributions and refinements of the results we state. 

\medskip

Let $V$ be a Euclidean vector space of dimension $d\geq 1$. Let $\Lam$ be a lattice in $V$, that is a discrete $\Z$-submodule of rank $d$. This means that $\Lam  = \Z e_1 + \ldots + \Z e_d$ for some basis $(e_1,\ldots,e_d)$ of $V$. Of course the most important example is $\Z^d$ if $V= \R^d$.

Let $\calC \subset V$ be a convex compact subset, symmetric with respect to the origin (that is, $-x \in \calC$ for any $x\in \calC$), with positive volume. Then  $\calC$ contains a small open ball around the origin. For any $\lam \in \R$ we let $\lam \calC$ denote the set of all $\lam x$ with $x\in \calC$.

Let $k\in\und$. Then the set of all positive real numbers $\lam$ such that $\lam \calC \cap \Lam$ contains $k$ linearly independent points is a closed interval $[\lamk,+\infty[$: the positive real number $\lamk$ defined in this way is called the  {\em $k$-th successive minimum} of $\calC$ with respect to $\Lam$.
In particular, $\lamu$ is the least positive  real number $\lam$ such that $ \lam \calC \cap \Lam \neq \{0\}$ (since $\calC$ is compact, it is bounded so that $ \lam \calC \cap \Lam  = \{0\}$ if $\lam>0$ is sufficiently small, because $\Lam$ is discrete). On the other hand, $\lamd$ is the least $\lam > 0 $ such that  $ \lam \calC \cap \Lam$ contains a basis of $V$. The definition of successive minima yields easily the following properties:
\begin{equation}\label{eqordre}
0 < \lamu \leq \ldots \leq \lamk \leq \ldots \leq \lamd, 
\end{equation}
\begin{equation}\label{eqinclus}
\lamk \geq \lamkpr
\end{equation}
for any $1\leq k \leq d$, if $\calC \subset \calCpr$ and $\Lam \subset \Lampr$, and
\begin{equation}\label{eqmultcvx}
\lamkmuc = \mu^{-1} \lamk
\end{equation}
\begin{equation}\label{eqmultlam}
\lamkmulam = \mu  \lamk
\end{equation}
for any $1\leq k \leq d$ and any $\mu>0$. 

A basis of $V$ consisting of vectors of $\Lam $ is not always a $\Z$-basis of $\Lam$: the $\Z$-submodule of $\Lam$ it generates may have an index $\geq 2$ in $\Lam$. Therefore the definition of $\lamd$ does not imply the existence of a $\Z$-basis of $\Lam$ consisting in vectors of $  \lamd \calC   $. 
However it is known (see the remark after the corollary of Theorem VII in Chapter VIII of \cite{Cassels}) that there is a $\Z$-basis of $\Lam$ consisting in vectors of $d \lamd \calC  $ (and this will be the main argument in the proof of Proposition \ref{corutile}, in \S \ref{ssec3}). More generally, for any $k \in \und$ there exist linearly independent vectors $e_1,\ldots,e_k \in  k \lamk \calC \cap \Lam $ such that $\Z e_1 + \ldots + \Z e_k$ is saturated in $\Lam $ (that is, $\Span_\R(e_1,\ldots,e_k) \cap \Lam = \Z e_1 + \ldots + \Z e_k$), so that $e_1,\ldots,e_k$ are the first $k$ vectors of a basis of $\Lam$.

\bigskip

Our main tool in the proof of Theorem \ref{thdio} is a duality result, namely Theorem VI in Chapter VIII of \cite{Cassels}, \S 5. Let $\calC$ and $\Lam $ be as above. The polar (or dual) convex body $\calCet$ of $\calC$ is the set of all $x\in V$ such that, for any $y \in \calC$, $x\cdot y \leq 1$ (where $\cdot$ is the scalar product on $V$); see for instance Chapter IV, \S 3.3 of  \cite{Cassels}. Then $\calCet $ is also  a convex compact subset of $V$, symmetric with respect to the origin, with positive volume. The polar (or dual, or reciprocal) lattice $\Lamet $ of $\Lam$ is  the set of all $x\in V$ such that, for any $y \in \Lam$, $x\cdot y \in \Z$; see for instance Chapter I, \S 5 of  \cite{Cassels}. Then  Theorem VI in Chapter VIII of \cite{Cassels}, \S 5, asserts that
\begin{equation} \label{eqdualCassels}
1 \leq \lamk \lamdpumket \leq d!
\end{equation}
for any $k \in \und$. The constants are not important in the applications we have in mind, so this result means that $\lamk$ is essentially equal to $1/\lamdpumket$. With $d!^2$ instead of $d!$ in the upper bound (which does not matter for our purposes), it has been proved by Mahler \cite{Mahler1939} (see also \cite{MahlerCompound}) after an earlier result of Riesz \cite{Riesz}. More recent results show that the constant $d!$ can be in fact improved to $Cd(1+\log d)$ for some universal constant $C$ (see \cite{Ban96}, the optimal constant is expected to be $Cd$).

\bigskip

Let $\Lam$ be a lattice, and $(e_1,\ldots,e_d)$ be a $\Z$-basis of $\Lam$. Then the determinant of $(e_1,\ldots,e_d)$ with respect to an orthonormal basis of $F$    depends, in absolute value,  neither on  $(e_1,\ldots,e_d)$ nor on the chosen orthonormal basis. Its absolute value is the determinant of $\Lam$, denoted by $\det \Lam$. In general the theory of successive minima is introduced in $\R^d$, so that the orthonormal basis can be chosen to be the canonical one. If $\Lamet $ denotes the dual lattice as above, then $\det \Lamet = (\det \Lam)^{-1}$ (see Lemma 5 in Chapter I, \S 5 of  \cite{Cassels}).

\bigskip

Even though we won't use it  directly in this paper, it seems important to recall Minkowski's theorem 
$$\frac{2^d \det(\Lam)}{d! \vol(\calC)} \leq \prod_{k=1} ^d \lamk \leq \frac{2^d \det(\Lam)}{\vol(\calC)}$$
where $ \vol(\calC)$  is the volume of $\calC$ (see Chapter VIII, \S 1 of \cite{Cassels}); this result is used in the proof of \eqref{eqdualCassels}. Using \eqref{eqordre} it implies
$$\lamu^d \, \, \vol(\calC) \leq 2^d \det(\Lam).$$
Accordingly, if $ \vol(\calC) \geq  2^d \det(\Lam)$ then $\lamu \leq 1$ so that  $\calC \cap \Lam \neq \{0\}$: this consequence is known as Minkowski's first theorem on convex bodies.

\subsection{Proof of the  Diophantine results} \label{ssec3}

In this section we relate the functions $\pphk$ and $\ppsk$ defined in \S \ref{ssec1} to successive minima, and apply classical results recalled in \S \ref{ssec2} to prove Theorem \ref{thdio} stated  in \S \ref{ssec1}, and also that $\Psi_k(Q)$ exists for any $Q\geq 1$. At the end of this section we prove also Proposition \ref{corutile}, using the same tools.

With the notation of  \S \ref{ssec1}, for any $Q,K > 0$ we denote by $\cc{Q,K}$ the set of all $x\in F$ such that \eqref{eqphi} hold, that is
\[ \mathcal{C}(Q,K)=\{x\in F \; | \; |x|\leq K, \; |x_0\alpha-\alpha_0x|\leq Q^{-1}\}. \]
This is a compact convex subset, symmetric with respect to the origin, in the vector space $F$ equipped with the Euclidean structure induced from the canonical one on $\R^n$. It has a positive volume (see Lemma 4 in Appendix B of \cite{Casselsdio}). 

We let also $\Lam = \Z^n \cap F$. By definition, $F$ has a basis consisting in vectors of $\Q^n$. Multiplying these vectors by a common denominator of their coordinates yields a basis of $F$ consisting in vectors of $\Z^n$. Therefore the $\Z$-module $\Lam = \Z^n \cap F$ has rank at least $d$; since it is discrete, it has rank $d$ and it is a lattice in $F$.

Now for any $k \in\und$ and any $Q \geq 1 $ we have
\begin{equation} \label{eqlienphi}
\pphk(Q) = \inf\{K>0,\, \lam_k(\cc{Q,K},\Lam)\leq 1\},
\end{equation}
provided it is not the infimum of an empty set (and we shall prove this below).

To obtain an analogous property for $\ppsk$, for any $Q,K>0$ we denote by  $\ccpr{Q,K}$ the set of all $x\in F$ such that \eqref{eqpsi} hold, that is
\[ \ccpr{Q,K}=\{x\in F \; | \; |x|\leq Q, \; |x\cdot\al|\leq K^{-1}\}. \]
Then $\ccpr{Q,K}$ is also a compact convex subset, symmetric with respect to the origin, with positive volume; and we have for any $k\in\und$ and any $Q\geq \Cpsipr$:
\begin{equation} \label{eqlienpsi}
\ppsk(Q) = \sup\{K>0,\, \lam_k(\ccpr{Q,K},\Lam)\leq 1\}.
\end{equation}

We will deduce Theorem \ref{thdio} from \eqref{eqdualCassels} (see  \S \ref{ssec2}); with this aim in mind, we denote by $\Lamet$ the lattice dual to $\Lam$, and by $\ccet{Q,K}$   the convex body  dual to $\cc{Q,K}$. The main step is the following lemma.

\begin{lemma}\label{lemdio} 
We have 
$$\Lam \subset \Lamet \subset (\det \Lam)^{-2} \Lam$$
and, for any $Q,K>0$ such that $2 |\al | Q K \geq 1$,
$$(n+1)^{-1}|\al|\ccpr{Q,K} \subset \ccet{Q,K} \subset \ccpr{2|\al| Q, |\al| ^{-1} K  }.$$
\end{lemma}

In general $ \Lamet$ is not equal to $\Lam$, so that  $(\det \Lam)^{-2}$ can not be replaced with 1. However it can probably be replaced with another function of $\Lam$  (see the proof below). 

\begin{proof}[Proof of Lemma~\ref{lemdio}] Since $\Lam \subset \Z^n$, we have $x\cdot y \in \Z$ for any $x,y\in\Lam$ so that $\Lam \subset \Lamet$. Therefore  $\Lam$ is a sub-lattice of $\Lamet$, so that there exists a $\Z$-basis $(v_1,\ldots,v_d)$ of $\Lamet$ and positive integers $\ell_1,\ldots,\ell_d$ such that $(\ell_1v_1,\ldots,\ell_dv_d)$ is 
 a $\Z$-basis  of $\Lam$. Then $\ell_1\ldots\ell_d$ is the index of $\Lam$ in $\Lamet$, that is the absolute value of the determinant of  $(\ell_1v_1,\ldots,\ell_dv_d)$ with respect to the basis $(v_1,\ldots,v_d)$. Since $(\ell_1v_1,\ldots,\ell_dv_d)$  (respectively  $(v_1,\ldots,v_d)$) has  determinant equal (up to a sign) to $\det \Lam$ (respectively $\det \Lamet = (\det \Lam)^{-1}$) in an orthonormal basis of $F$, the index $\ell_1\ldots\ell_d$ is equal to $ (\det \Lam)^2$. Therefore all integers $\ell_j$ are divisors of $ (\det \Lam)^2$, and we have $ (\det \Lam)^2  \Lamet \subset  \Lam $.

Let $y\in  \ccet{Q,K} $, so that 
\begin{equation} \label{eqdemlemun}
 \forall x \in \cc{Q,K}, \quad | x \cdot y | \leq 1 
\end{equation}
by applying the definition of a dual convex body to both $x$ and $-x$, since $ \cc{Q,K}$ is symmetric with respect to 0. Since $|\al|^{-1}K \al = (|\al|^{-1}K\al_0,\ldots,|\al|^{-1}K\al_{n-1})\in  \cc{Q,K}$ because $\al \in F $, \eqref{eqdemlemun} yields $|y\cdot \al| \leq |\al| K^{-1}$. On the other hand, if $y\neq 0$ then $(2Q |y| \, |\al|)^{-1}y\in  \cc{Q,K}$ since $2 |\al | Q K \geq 1$, so that \eqref{eqdemlemun} yields $|y|^2 \leq y\cdot y \leq 2Q  |y| \, |\al| $ and therefore $|y| \leq 2   |\al| Q$. This concludes the proof that $\ccet{Q,K} \subset \ccpr{2|\al| Q, |\al| ^{-1} K  }$. 

Now let $y\in \ccpr{Q,K} $ and $x\in \cc{Q,K} $. Then we have
$$|\al_0 x\cdot y| = \Big| (\al_0 x - x_0 \al)\cdot y + x_0 \al \cdot y \Big| \leq n | \al_0 x - x_0 \al | \, |y| + |x| \, |\al \cdot y| \leq \frac{n}{Q}Q+K \frac{1}K=n+1,$$
so that $(n+1)^{-1}|\al|y \in  \ccet{Q,K} $ since $|\al| = |\al_0|$, thereby concluding the proof of Lemma \ref{lemdio}.
\end{proof}

\begin{proof}[Proof of Theorem \ref{thdio}] First of all, let us notice that for any positive real number $\mu$ and any $Q,K> 0$ we have
\begin{equation} \label{eqhom}
\begin{cases}
\mu \cc{Q,K} = \cc{\mu^{-1} Q,\mu K}, \\
\mu \ccpr{Q,K} = \ccpr{ \mu Q ,\mu^{-1}K}.
\end{cases}
\end{equation}
Using \eqref{eqmultcvx}, these equalities imply
\begin{equation} \label{eqmult}
\begin{cases}
\lam_k \left( \cc{\mu^{-1}Q,\mu K}, \Lam\right) = \mu^{-1}  \lam_k \left(  \cc{Q,K}, \Lam\right), \\
\lam_k \left(  \ccpr{ \mu Q ,\mu^{-1}K}  , \Lam\right) = \mu^{-1}  \lam_k \left(  \ccpr{Q,K}, \Lam\right)
\end{cases}
\end{equation}
for any $k\in\und$ and any $\mu > 0$.

We shall prove now, at the same time, that $\Psi_k(Q)$ exists and that Theorem \ref{thdio} holds. Let $k\in\und$, and $Q \geq \max(1, (n+2)\Cpsipr |\al|^{-1})$. If $\Psi_k(Q)$ exists, for any   $\eps> 0$  sufficiently small  \eqref{eqlienphi} yields
\begin{equation*}
\begin{cases}
\lambda_k\left(\cc{ Q, (1-\eps) \pphk(Q)}, \Lam\right) > 1, \\
\lambda_k\left(\cc{ Q, (1+\eps) \pphk(Q)}, \Lam\right) \leq 1.
\end{cases}
\end{equation*}
If $\Psi_k(Q)$ does not exist, then the first inequality is valid  when $  (1-\eps) \pphk(Q)$ is replaced with any positive real number, and in what follows the first inequality in each pair will still be valid in this case.

The main duality result from geometry of numbers, namely \eqref{eqdualCassels}, enables one to deduce that 
\begin{equation*}
\begin{cases}
\lambda_{d+1-k}  \left(\ccet{ Q, (1-\eps) \pphk(Q)}, \Lamet\right) < d!, \\
\lambda_{d+1-k}\left(\ccet{ Q, (1+\eps) \pphk(Q)}, \Lamet\right) \geq 1
\end{cases}
\end{equation*}
Using \eqref{eqinclus} and the fact that $2|\al| (1-\eps)\pphk(Q)Q\geq  (n+2)\Cpsipr \Psi_k(Q) \geq 1$ for $\eps \leq 1/2$ (because $\pphk$ takes positive integer values), Lemma~\ref{lemdio} yields
\begin{equation*}
\begin{cases}
\lambda_{d+1-k}  \left(\ccpr{ 2|\al| Q,  (1-\eps)  |\al| ^{-1} \pphk(Q) } , (\det\Lam)^{-2} \Lam \right) < d!, \\
\lambda_{d+1-k}\left( (n+1)^{-1}|\al| \ccpr{ Q, (1+\eps) \pphk(Q)}, \Lam \right) \geq 1.
\end{cases}
\end{equation*}
Using \eqref{eqmultlam} and \eqref{eqhom} we obtain
\begin{equation*}
\begin{cases}
\lambda_{d+1-k}  \left(\ccpr{ 2|\al| Q,   (1-\eps)|\al|^{-1} \pphk(Q)   },  \Lam \right) < (\det\Lam)^{ 2}d!, \\
\lambda_{d+1-k}\left(  \ccpr{  (n+1)^{-1}|\al| Q , (1+\eps) (n+1) |\al|^{-1} \pphk(Q)}, \Lam \right)\geq 1.
\end{cases}
\end{equation*}
Now \eqref{eqmult} implies
\begin{equation*}
\begin{cases}
\lambda_{d+1-k}  \left(\ccpr{ 2|\al|  (\det\Lam)^{ 2}d! Q, ((\det\Lam)^{ 2}d! |\al|)^{-1}(1-\eps) \pphk(Q)},  \Lam \right) < 1, \\
\lambda_{d+1-k}(  \ccpr{((1+\eps)(n+1))^{-1}|\al| Q, (1+\eps)^2 (n+1) |\al|^{-1} \pphk(Q)}, \Lam )\geq 1+\eps.
\end{cases}
\end{equation*}
Since for any $Q'\geq 1$, the set of $K > 0$ such that $\lambda_{d+1-k}(\ccpr{Q',K},\Lam)\leq 1$  is an interval between 0 and $\ppsdmk(Q')$, we obtain
\begin{equation*}
\begin{cases}
((\det\Lam)^{ 2}d! |\al|)^{-1}(1-\eps) \pphk(Q)  \leq  \ppsdmk (    2|\al|  (\det\Lam)^{ 2}d! Q), \\
(1+\eps)^2 (n+1) |\al|^{-1} \pphk(Q)  \geq \ppsdmk (((1+\eps)(n+1))^{-1}|\al| Q).
\end{cases}
\end{equation*}
The first inequality if false when  $  (1-\eps) \pphk(Q)$ is replaced with a sufficiently large real number: this concludes the proof that $\Psi_k(Q)$  exists. Moreover, 
by letting $\eps$ tend to zero (so that $(1+\eps)(n+1) \leq n+2$), these inequalities  conclude also the proof of Theorem \ref{thdio}.
\end{proof}

\begin{proof}[Proof of Proposition~\ref{corutile}] 
Using \eqref{eqmult} and the fact that the infimum in the definition of $\Psi_d(Q)$ is attained, we have
$$d \lambda_d( \cc{d^{-1}Q, d\Psi_d(Q)}, \Lambda) =  \lambda_d( \cc{Q,  \Psi_d(Q)}, \Lambda) \leq 1$$
so that there exists a  basis of $\Lam$ contained in   $\cc{d^{-1}Q, d\Psi_d(Q)}  $ (using   the remark after the corollary of Theorem VII in Chapter VIII of \cite{Cassels}, recalled in \S \ref{ssec2}). Let us denote by $x_1,\dots,x_d$ this basis, and fix $j\in \{1,\dots,d\}$. Since $x_j\in\cc{d^{-1}Q, d\Psi_d(Q)} $, if we write $x_j=(x_{j,0}, \dots, x_{j,n-1})$, we have
\begin{equation}\label{incorutile}
\begin{cases}
| x_{j} | \leq d\Psi_d(Q),  \\ 
| x_{j,0} \alpha - \alpha_0 x_{j}| \leq dQ^{-1}. 
\end{cases}
\end{equation}
We claim that $x_{j,0}\neq 0$. Indeed, if $x_{j,0}=0$, then the second inequality in~\eqref{incorutile} gives $|x_{j,i} \alpha_0 | \leq dQ^{-1}$ for all $i \in \{1,\dots,n-1\}$. As $|\alpha|=|\alpha_0|$, this implies that $|x_{j,i}| \leq d(|\alpha|Q)^{-1}$ but since $Q>d|\alpha|^{-1}$ by assumption, this gives $|x_{j,i}|<1$. Now $x_{j,i}\in\Z$, so this necessarily implies $x_{j,i}=0$ and together with $x_{j,0}=0$, the vector $x_j$ has to be zero. This contradicts the fact that $x_j$ is an element of a basis, and proves the claim. Now we can define $\omega_j=x_{j,0}^{-1}\alpha_0 x_j$, then $\omega_j$ is $T_j$-periodic with $T_j=x_{j,0}\alpha_0^{-1}$, and replacing $x_j$ by $-x_j$ if necessary, we can assume $T_j>0$. Since $T_j\omega_j=x_j$, this proves the first part of the statement. Then from the second inequality in~\eqref{incorutile}, we have
\[ |T_j\alpha-T_j\omega_j|=| x_{j,0} \alpha_0^{-1} \alpha - x_{j}|\leq d(|\alpha|Q)^{-1}, \]
hence
\[ |\alpha-\omega_j|\leq d(|\alpha|T_jQ)^{-1} \]
and from the first inequality in~\eqref{incorutile},
\[ |\alpha|^{-1}\leq T_j \leq |\alpha|^{-1}d\Psi_d(Q),  \]
which gives the second part of the statement. This concludes the proof.
\end{proof}

\subsection{Minimal families} \label{subsecminifam}

In this section we prove Proposition \ref{proppsietpsiprime} by  interpreting the functions  $\Psi_k$ and $\Psi'_k$ in the spirit of the sequence of minimal points introduced by Davenport and Schmidt (see \cite{DS67}, \S 3 or \cite{DS}, \S 3 and \S 7), even though a direct proof would be possible. This interpretation seems interesting in itself, and it will be useful to compute these functions explicitly in the case $n=2$ (see \S \ref{subsecunnb}).

\begin{proof}[Proof of Proposition  \ref{proppsietpsiprime}]
Let $k \in \und$, and denote by $\Ek$ the (countable) set of all families $\es  = (e_1,\ldots,e_k)$ consisting in $k$ linearly independent vectors in $\Z^n\cap F$. Letting $e_p  = (e_{p,0},\ldots,e_{p,n-1})$ for any $p \in\unk$, we put 
$$ | \es |  = \max_{1\leq p \leq k}   \max_{0\leq j \leq n-1} | e_{p,j}|.$$
We fix an ordering on $\Ek$ (given by a bijective map of $\Ek$ to $\N$) such that if $|\es| < |\es'|$, then $\es \mbox{ comes before } \es'$. To study the function $\Psi_k$, we let
$$\eta(\es) = \max_{1\leq p \leq k} \max_{1\leq j \leq n-1} | e_{p,0}\al_j - e_{p,j} \al_0|.$$
Notice that for any $\es\in\Ek$, $|\es|$ is a positive integer and $\eta(\es)$ is a positive real number; there exist values of $\eta(\es)$ arbitrarily close to 0 (because $\Psi_k(Q)$  exists for any $Q\geq 1$).

Recall from  Definition \ref{defpsi} (\S \ref{ssec1}) that $\Cpsipr$ is the least value of $|\es|$ with $\es\in\Ek$. 
For any real $X \geq \Cpsipr$ we consider the  set of $\es\in\Ek$ such that $|\es| \leq X$. Among this finite non-empty set, we consider the subset consisting in all $\es$ for which $\eta(\es)$ takes its minimal value. We call {\em minimal family corresponding to $X$} the element in this subset which comes first  in the ordering  we have fixed on $\Ek$. The choice of this ordering is not important, but choosing one in advance enables us  to make consistent choices (as $X$ varies): for instance, changing some $e_p$ into $-e_p$, or making a permutation of $e_1,\ldots,e_k$, leaves $|\es|$ and $\eta(\es)$ unchanged. Anyway the most important for us will be $|\es|$ and $\eta(\es)$ for minimal families $\es$, and these values do not depend on the ordering we choose.

Let $\Fk$ be the set of all minimal families  $\es$ (corresponding to some  $X \geq \Cpsipr$) such that  $\eta(\es)\leq 1$. Then $\Fk$ is infinite, countable, and for any distinct $\es, \es'\in\Fk$ we have $|\es|\neq |\es'|$. Let $( \el  )_{\ell \geq 1}$ denote the sequence of all elements of $\Fk$, ordered in such a way that the integer  sequence $(|\el |)_{\ell \geq 1}$ increases to $+\infty$. Then   $ \el $ is the  minimal family corresponding to all  $X$ in the range $|\el| \leq X < |\elpu|$, and   the real sequence $(\eta(\el))_{\ell \geq 1}$ decreases to 0.

For any $Q$ such that $\eta(\el)^{-1} < Q \leq \eta(\elpu)^{-1} $ with $\ell \geq 1$, we have $\Psi_k(Q) = |\elpu|$ because the infimum in the definition of $\Psi_k(Q)$ is attained by the family $\elpu$. In the same way, we have   $\Psi_k(Q) = |\eun|$ for any $Q$ such that $1\leq Q \leq \eta(\eun)^{-1}$. Therefore $\Psi_k$ is constant on each interval $] \eta(\el)^{-1}, \eta(\elpu)^{-1} ]$ (and also on $[1,  \eta(\eun)^{-1}]$); it is left-continuous, and its values are positive integers.

\bigskip

A similar interpretation holds for $\Psi'_k$, letting
$$\eta'(\es) = \max_{1\leq p \leq k}| e_p \cdot \al | $$
for any $\es = (e_1,\ldots,e_k)\in\Ek$. 
For any real $X \geq \Cpsipr$ we consider   the  set of $\es\in\Ek$ such that $|\es| \leq X$. Among this finite set, we focus on  the subset consisting in all $\es$ for which $\eta'(\es)$ takes its minimal value. Within this subset,   we call {\em minimal family corresponding to $X$} the one   which comes first  in the ordering we have fixed on $\Ek$. Of course   the minimal families here are not the same as the ones above (which are defined in terms of   $\eta$). Since the minimal families (corresponding to some $X \geq \Cpsipr$) make up an infinite countable set such that  $|\es| \neq |\es'|$ as soon as $\es\neq \es'$, they can be ordered in a sequence  $(|\eprl |)_{\ell \geq 1}$ with the following properties:  the integer  sequence $(|\eprl |)_{\ell \geq 1}$ increases to $+\infty$;  the real sequence $(\eta'(\eprl))_{\ell \geq 1}$ decreases to 0; $\eprl$ is the  minimal family corresponding  to all $X$ in the range $|\eprl|   \leq X < |\eprlpu|$. 

For any $Q$ such that $|  \eprl |  \leq Q<|  \eprlpu |  $ with $\ell \geq 1$, we have $\Psi'_k(Q) = \eta(  \eprl)^{-1}$ because the supremum in the definition of $\Psi'_k(Q)$ is attained by the family $\eprl$.
 Therefore $\Psi'_k$ is constant on each interval $[ |  \eprl |  ,  |  \eprlpu |  [ $; it is right-continuous, and its  points of discontinuity are positive integers. 
 This concludes the proof of Proposition \ref{proppsietpsiprime}.
\end{proof}

\subsection{The case $n=2$} \label{subsecunnb}

Let us focus now on the case of just one number, in which the functions $\Psi_k$ and $\Psi'_k$ can be made explicit in terms of continued fractions. Eventhough  this section won't be used in the rest of the paper, we think it is   an instructive example. In precise terms, we take $n=2$, $\al_0 = 1$ and $\al_1 = \xi$ with $\xi\in\R\setminus\Q$ and  $0 <  \xi <1$. We have $d=2$ and $F = \R^2$.

\newcommand{\pej}{|q_j \xi - p_j|}
\newcommand{\pejpu}{|q_{j+1} \xi - p_{j+1}|}
\newcommand{\pejmu}{|q_{j-1} \xi - p_{j-1}|}
\newcommand{\pejt}{|q_{j,t} \xi - p_{j,t}|}
\newcommand{\pejtpu}{|q_{j,t+1} \xi - p_{j,t+1}|}
\newcommand{\pejtmu}{|q_{j,t-1} \xi - p_{j,t-1}|}
\newcommand{\pejajpumu}{|q_{j,a_{j+1}- 1} \xi - p_{j,a_{j+1}-1}|}
\newcommand{\peajpusurtrois}{|q_{j,a_{j+1}/3}\xi - p_{j,a_{j+1}/3}|}

We denote by $p_j/q_j$ the $j$-th convergent in the continued fraction expansion of $\xi$, with $p_0=0$, $q_0=1$, $p_{j+1} = a_{j+1} p_{j } + p_{j-1}$ and $q_{j+1} = a_{j+1} q_{j } + q_{j-1}$ for any $j \geq 1$ where $a_j$ is the $j$-th partial quotient (see for instance \cite{Schmidt} or \cite{HW}). We also need to consider the semi-convergents defined for any $j \geq 1$ and any integer $t$ with $0\leq t \leq a_{j+1}$ by
$$q_{j,t} = t q_j + q_{j-1}, \quad  p_{j,t} = t p_j + p_{j-1}.$$
Notice that 
$$q_{j,0} = q_{j-1}, \, p_{j,0} = p_{j-1}, \quad q_{j,a_{j+1}} = q_{j+1}, \, p_{j,a_{j+1}} = p_{j+1},$$
and that for any $t \in \{0,\ldots,a_{j+1}\}$:
\begin{equation} \label{eqepsj}
| q_{j,t}\xi - p_{j,t}|  = | q_{j-1}\xi-p_{j-1}| - t   | q_j   \xi-p_j| = (a_{j+1}+\eps_j-t) \pej > 0
\end{equation}
since the sign of $q_{n }\xi-p_{n }$ is that of $(-1)^n$; here $\eps_j\in]0,1[$ is such that $\eps_j \pej = \pejpu$. In particular this quantity is a decreasing function of $t\in\{0,\ldots,a_{j+1}\}$, and we have 
$| q_{j,t}\xi - p_{j,t}|  > | q_{j }\xi - p_{j }|$ if $t < a_{j+1}$.
 Let us recall also that 
$$q_{j+1}<  \pej^{-1} < q_{j+1}+q_j=q_{j+1,1} < 2 q_{j+1}.$$
These relations will be used repeatedly below, without explicit reference.

To compute the functions $\Psi_k$ and $\Psi'_k$, we recall that 
for any $Q \geq 1$ there is an integer $j$, and only one, such that $q_j \leq  Q < q_{j+1}$. There is also a unique $t$ such that $0\leq t <  a_{j+1}$ and $q_{j,t} \leq  Q < q_{j,t+1}$.

\begin{proposition}\label{propexplipsipsipr}
For any $j$ and any $Q\geq 1$ we have:
\[ \Psi_1(Q) = q_j, \quad \pejmu ^{-1}< Q \leq \pej^{-1},\]
\begin{equation*}
\Psi_2(Q)=
\begin{cases}
q_{j,t+1}, \quad \pejt  ^{-1}< Q \leq \pejtpu^{-1}, \quad 0\leq t \leq a_{j+1}-2, \\
q_{j+1}, \quad \pejajpumu   ^{-1}< Q \leq \pej ^{-1}, 
\end{cases}
\end{equation*}
and
\[ \Psi'_1(Q) = \pej^{-1}, \quad q_j\leq  Q < q_{j+1},\]
\begin{equation*}
\Psi'_2(Q)=
\begin{cases}
\pejmu^{-1}, \quad q_j\leq  Q < q_{j, 1}, \\
\pejt^{-1}, \quad q_{j,t}  \leq  Q < q_{j,t+1}, \quad 1\leq t \leq a_{j+1}-1. 
\end{cases}
\end{equation*}
These functions satisfy the following inequalities for any $Q\geq 1$:
\[ \frac13 \Psi'_2(Q/3) < \Psi_1(Q)< \Psi'_2(Q), \quad \frac13 \Psi'_1(Q/3) < \Psi_2(Q) < \Psi'_1(Q).\]
\end{proposition}

The lower bound for $\Psi_k$ is essentially the same as the one provided by Theorem \ref{thdio} in this case (see the remark after the statement of this result). Of course the proof is different (and we find it instructive) because it relies on the explicit determination of the functions. By following this proof it is not difficult to construct examples of numbers $\xi$ for which this lower bound is essentially an equality for infinitely many $Q$.

On the other hand, the upper bound on $\Psi_k$ refines upon the one provided by Theorem \ref{thdio}, and can also be seen to be essentially an equality for infinitely many $Q$, if $\xi$ is properly chosen.

\begin{proof}[Proof of Proposition \ref{propexplipsipsipr}]  We compute $\Psi_k$ and $\Psi'_k$ by determining the sequences of minimal families defined in \S \ref{subsecminifam}.   Let us start by computing $\Psi_1$ and $\Psi'_1$, that is by taking   $k=1$. For a non-zero vector $\es = e_1 = (e_{1,0},e_{1,1}) \in \Z^2$ we have $|\es| = \max (| e_{1,0}|,|e_{1,1}|)$, 
$$\eta(\es) = |e_{1,0}\xi-e_{1,1}|, \quad \etapr(\es) = |e_{1,0}+ e_{1,1} \xi |.$$
It is a classical property of continued fractions that $((q_\ell, p_\ell))_{\ell \geq 0} $ is the sequence of minimal families defined in terms of $\eta$ and related to $\Psi_1$, if the ordering on $\Eun$ is chosen properly. In the same way $((p_\ell, -q_\ell))_{\ell \geq 0} $ is the  one    related to $\Psi'_1$. These facts lead immediately to the expressions for $\Psi_1$ and $\Psi'_1$ in Proposition \ref{propexplipsipsipr} (see the end of  \S \ref{subsecminifam}). 

\smallskip

The case $k=2$ is slightly more complicated.  
We let 
$$\es_{j,t} = ( (q_j,p_j),  (q_{j,t}, p_{j,t})).$$
In particular   $\es_{j,0}$ consists in two consecutive convergents; so does $\es_{j,a_{j+1}}$, which is equal to $  \es_{j+1,0}$ up to permuting the two vectors. The determinant of the vectors ${\tiny \left(\begin{array}{c} q_{j} \\ p_j \end{array}\right)}$ and  ${\tiny \left(\begin{array}{c} q_{j,t} \\ p_{j,t} \end{array}\right)}$ is equal to that of ${\tiny \left(\begin{array}{c} q_{j} \\ p_j \end{array}\right)}$ and  ${\tiny \left(\begin{array}{c} q_{j-1} \\ p_{j-1} \end{array}\right)}$, so that it is $\pm 1$: the vectors in the family $\es_{j,t}$ make up a basis of $\Z^2$.

We have
$$|\es_{j,t} |  = q_{j,t} =   t q_j + q_{j-1}$$   
if $t \geq 1$ and $|\es_{j,0} |  =  q_j$, whereas
$$  \eta(\es_{j,t} ) =\pejt =  | q_{j-1}\xi-p_{j-1}| - t   | q_{j }\xi-p_{j }| > 0$$ 
if $t \leq a_{j+1}-1$ and $\eta(\es_{j,a_{j+1}} ) = \pej$. 

Let us prove that (up to changing the ordering we have fixed on $\Ede$) the sequence $(\el)_{\ell \geq 1} $ of minimal families constructed in \S \ref{subsecminifam} is exactly the sequence of families $\es_{j,t}$ with $j\geq 1$ and $0\leq t \leq a_{j+1}-1$, indexed in such a way that $|\es_{j,t}|$ is increasing. With this aim in view, let $X\geq 1$ and denote by $\es=(e_1,e_2)$ the minimal family corresponding to $X$. Let $j\geq 1$ and $t\in\{0,\ldots,a_{j+1}-1\}$ be such that $|\es_{j,t}|\leq X < |\es_{j,t+1}|$ (that is, $\max(q_j,q_{j,t})\leq X < q_{j,t+1}$). Let $e = (q, p)$ denote either $e_1$ or $e_2$. We have $|q| \leq X < q_{j+1}$ and $|q\xi-p| \leq \eta(\es_{j,t})\leq \pejmu$ so that:
\begin{eqnarray*}
\left| \det \left[\begin{array}{cc} q & q_j \\ p & p_j \end{array} \right] \right| 
&=&
\left| \det \left[\begin{array}{cc} q & q_j \\ q\xi - p & q_j\xi - p_j \end{array} \right] \right| \\
&\leq& |q| \cdot \pej + q_j |q\xi-p| < q_{j+1} \pej + q_j \pejmu < 2
\end{eqnarray*}
so that this determinant is equal to $-1$, 0, or 1. Since $\es_{j,t}$ is a basis of $\Z^2$, this yields $e = (q,p)  =\al (q_j,p_j) + \beta(q_{j,t}, p_{j,t})$ with $\al\in\Z$ and $\beta\in\{-1,0,1\}$. Changing $(q,p)$ into $(-q,-p)$ if necessary (which follows from a suitable change in the ordering on $\Ede$), we may assume that $e = \al (q_j,p_j)$ with $\al \geq 1$ if $\beta= 0$, and $\beta=1 $ otherwise. In the latter case we have $q = q_{j-1} + (t+\al)q_j$ with $\al \leq 0$ (since $|q| \leq X < q_{j,t+1}$),  $t+\al \geq 0$ (since $|q\xi-p|  \leq  \pejmu$, using the fact that $q_{j-1}\xi-p_{j-1}$ and $q_j\xi-p_j$ have opposite signs), and finally $\al = 0$ because $|q\xi-p|\leq \eta(\es_{j,t}) = \pejt$. Therefore we have proved that (up to permuting the two vectors $e_1$ and $e_2$ which make up the family $\es$) we have $e_1 =  \al (q_j,p_j)$ with $\al \geq 1$ and $e_2 =   (q_{j,t},p_{j,t})$. Since $\eta(\es) \leq \eta(\es_{j,t})$ we deduce that  $\eta(\es) = \eta(\es_{j,t})$. Now $\es$ comes before $\es_{j,t}$ in the ordering on $\Ede$ (by definition of a minimal family), so that $|\es| \leq |\es_{j,t}| $ (thanks to the assumption we make, throughout the paper, on the orderings we choose on $\Ek$). If $t=0$ this implies $\al = 1$ and $\es = \es_{j,t}$. If $t \geq 1$ then $\al$ might be different from 1, but since $|\es| = |\es_{j,t}|$ and $\eta(\es) = \eta(\es_{j,t})$ we may change the ordering on $\Ede$ in such a way that $\es_{j,t} $ comes first, so that $\es = \es_{j,t}$ in this case too. This concludes the proof that $\es_{j,t}$ is the minimal family corresponding to $X$. 

The values of $\Psi_2(Q)$ follow immediately (see the end of \S \ref{subsecminifam}). The proof is similar for $\Psi'_2(Q)$: the  sequence  of minimal families  defined in terms of $\eta'$ and corresponding to $\Psi'_2$ is given by the families $((p_j,-q_j),(p_{j,t}, -q_{j,t}))  $ if the ordering on $\Ede$ is appropriate.

\bigskip

To conclude the proof of  Proposition \ref{propexplipsipsipr}, let us compare $\Psi_k$ and $\Psi'_{3-k}$, starting with the inequality   $\Psi_1(Q) < \Psi'_2(Q)$. Let $j$ be such that $\pejmu^{-1}< Q \leq \pej^{-1}$; then $\Psi_1(Q) = q_j$. Since $\Psi'_2$ is non-decreasing and $\pejmu^{-1}<q_{j,1}$, we may assume that  $\pejmu^{-1}< Q<q_{j,1}$ so that $\Psi'_2(Q)= \pejmu^{-1}$ and the conclusion follows.

Let us prove in the same way that $\Psi_2(Q) <  \Psi'_1(Q)$. Let $j$ be such that $q_j \leq Q  < q_{j+1}$, so that  $\Psi'_1(Q) = \pej^{-1}$. Since $Q < q_{j+1} \leq \pej^{-1}$ and $\Psi_2$ is non-decreasing, we have $\Psi_2(Q)\leq \Psi_2(\pej^{-1})=q_{j+1} <  \Psi'_1(Q)$.

The remaining two inequalities are slightly less straightforward. We shall begin with the following facts:
\begin{equation}\label{equn1}
\peajpusurtrois ^{-1}  <  3 q_j
\end{equation}
and
\begin{equation}\label{equn2}
q_{j, (a_{j+1}+1)/3} > \frac13 \pej^{-1},
\end{equation}
where $q_{j,t} $ and $p_{j,t}$ are defined by 
$  t q_j + q_{j-1}$   and  $  t p_j + p_{j-1} $ even if $t$ is not an integer. 

To prove   \eqref{equn1}, we write using \eqref{eqepsj}:
$$3q_j \peajpusurtrois = q_j (2a_{j+1}+3\eps_j)\pej >  2 q_j ( a_{j+1}+ \eps_j)\pej =  2 q_j  \pejmu > 1,$$
whereas   \eqref{equn2} follows from
$$\pej^{-1} < q_{j+1}+q_j=q_{j-1}+(a_{j+1}+1)q_j< 3\left(q_{j-1}+\frac{a_{j+1}+1}3 q_j\right).$$
Let us deduce now that $\frac13 \Psi'_2(Q/3) <  \Psi_1(Q)$ for any $Q$. Let $j$ be such that $\pejmu^{-1} < Q \leq \pej^{-1}$, so that $\Psi_1(Q)=q_j$. Since $\Psi'_2$ is non-decreasing we may assume $Q= \pej^{-1}$. Let $t$ denote the integer part of $a_{j+1}/3$. Then  \eqref{equn2} yields $Q/3 < q_{j,t+1}$ so that, using  \eqref{equn1} and the fact that $\Psi'_2$ is non-decreasing:
$$\Psi'_2(Q/3) \leq \pejt^{-1} <  3q_j=3\Psi_1(Q)$$
if $t\geq 1$, and 
$$\Psi'_2(Q/3)\leq \pejmu^{-1} < 2q_j=2\Psi_1(Q)$$
if $t=0$. This concludes the proof that $\frac13 \Psi'_2(Q/3) < \Psi_1(Q)$.

Let us prove now, along the same lines,  that $\frac13 \Psi'_1(Q/3) <  \Psi_2(Q)$. Let $j$ be such that $q_j \leq Q/3 < q_{j+1}$; then $ \Psi'_1(Q/3) = \pej^{-1}$. Since $\Psi_2$ is non-decreasing we may assume that $Q= 3q_j$. Letting $t$ denote the integer part of $a_{j+1}/3$,    \eqref{equn1} yields $ \pejt^{-1} < Q$. If $t \leq a_{j+1}-2$ we obtain  $\Psi_2(Q)\geq q_{j,t+1}> \frac13 \Psi'_1(Q/3)$ using \eqref{equn2}; if $t=a_{j+1}-1$ then $\Psi_2(Q) \geq q_{j+1} > \frac12 \Psi'_1(Q/3)$. In both cases this concludes the proof of Proposition \ref{propexplipsipsipr}. 
\end{proof}

\section{Applications to the KAM and Nekhoroshev theorems}\label{sec3}

Based on the results of the previous section, we describe now a new approach to the perturbation theory for quasi-periodic solutions. For simplicity, we will restrict to perturbations of constant vector fields on the $n$-dimensional torus, but clearly the approach can be extended to other situations. We will give stability results both in ``finite" and ``infinite" time, which correspond respectively to the construction of a ``partial" normal form, for any $1\leq d \leq n$, and to an ``inverted" normal form for $d=n$. The first result corresponds to a Nekhoroshev type theorem, while the second result is the KAM theorem for vector fields.

We state our main stability results, Theorem~\ref{thmfini} and Theorem~\ref{thminfini}, in \S\ref{ss30}. Then in \S\ref{ss31} we prove a periodic averaging result (Lemma~\ref{one}), which will be the only analytical tool in our proofs. A quasi-periodic averaging result (Lemma~\ref{multi}) will then be easily obtained from the periodic averaging result through approximation by periodic vectors (Proposition~\ref{corutile} in \S\ref{subseccordio}), and this constitutes the main novelty in our proofs. Finally, Theorem~\ref{thmfini} and Theorem~\ref{thminfini} will be proved respectively in \S\ref{ss33} and \S\ref{ss34} by applying inductively the quasi-periodic averaging lemma: the proof of Theorem~\ref{thmfini} consists of a finite iteration and is direct while the proof of Theorem~\ref{thmfini} relies on an infinite induction and is more implicit. 

\subsection{Main results}\label{ss30}

Let us first describe the setting. Let $n\geq 2$, $\T^n=\R^n/\Z^n$ and $\T_{\C}^{n}=\C^n/\Z^n$. For $z=(z_1,\dots,z_n)\in \C^n$, we define $|z|=\max_{1\leq i \leq n}|z_i|$, and given $s>0$, we define a complex neighbourhood of $\T^n$ in $\T_{\C}^{n}$ by 
\[ \T^n_s=\{\theta\in \T_{\C}^{n}\; | \; |\mathrm{Im}(\theta)|< s\}. \]
Given $\alpha\in\R^n\setminus\{0\}$, we will consider bounded real-analytic vector fields on $\T^n_{s}$, of the form
\begin{equation}\label{H}
X=X_\alpha+P, \quad X_\alpha=\alpha, \quad |P|_{s}=\sup_{z\in \T^n_s}|P(z)|\leq \varepsilon. \tag{$*$}
\end{equation}
By real-analytic, we mean that the vector field is analytic and is real valued for real arguments. For any such vector field $Y$, we shall denote by $Y^t$ its time-$t$ map for values of $t\in\C$ which makes sense, and given another vector field $Z$, we denote by $[Y,Z]$ their Lie bracket. Moreover, for a real-analytic embedding $\Phi : \T^n_{r} \rightarrow \T_{\C}^{n}$, $r\leq s$, and for a real-analytic vector field $Y$ which is well-defined on the image of $\Phi$, we let $\Phi^*Y$ be the pull-back $Y$, which is well-defined on $\T^n_{r}$.

Finally, if $d$ is the number of effective frequencies of $\alpha$, we denote by $\Psi_{\alpha}=\Psi_d$ the function defined in \S\ref{ssec1}, Definition~\ref{defphi}. From Proposition~\ref{proppsietpsiprime}, $\Psi_\alpha$ is left-continuous and non-decreasing, hence the function
\begin{equation}\label{delta}
\Delta_{\alpha} : [1,+\infty[ \rightarrow [1,+\infty[, \quad \Delta_{\alpha}(Q)=Q\Psi_{\alpha}(Q)
\end{equation}
is left-continuous and increasing. Therefore it has a generalized inverse
\begin{equation}\label{deltastar}
\Delta_{\alpha}^* : [1,+\infty[ \rightarrow [1,+\infty[, \quad \Delta_{\alpha}^*(x)=\sup\{Q \geq 1\; | \; \Delta_{\alpha}(Q)\leq x\} 
\end{equation}
which satisfies $\Delta_{\alpha}^*(\Delta_{\alpha}(Q))=Q$ and $\Delta_{\alpha}(\Delta_{\alpha}^*(x))\leq x$. Moreover, $\Delta_{\alpha}^*$ is both non-decreasing and continuous.

\bigskip
 
Our first result is the following.

\begin{theorem}\label{thmfini}
Let $X=X_\alpha+P$ be as in~(\ref{H}), define 
\[ Q(\varepsilon)=\Delta_{\alpha}^*\left((2\varepsilon)^{-1}\right), \quad \kappa_{d,\alpha}=2^6d^2(2^d-1)|\alpha|^{-1},\] 
and assume that
\begin{equation}\label{thr0}
Q(\varepsilon)>\max(1,d|\alpha|^{-1}), \quad Q(\varepsilon)\geq \kappa_{d,\alpha}s^{-1}.
\end{equation}
Then there exists a real-analytic embedding $\Phi : \T^n_{s/2} \rightarrow \T^n_{s}$ such that
\[ \Phi^*X=X_\alpha+N+R, \quad [X_\alpha,N]=0,  \]
with the estimates $|\Phi-\mathrm{Id}|_{s/2}<d^2|\alpha|^{-1}Q(\varepsilon)^{-1}$ and
\[ |N|_{s/2}<2\varepsilon, \quad |R|_{s/2}< 2\varepsilon \exp\left(-(\ln2)\kappa_{d,\alpha}^{-1}sQ(\varepsilon)\right). \] 
\end{theorem}

The theorem states that the perturbed vector field $X=X_\alpha+P$ can be analytically conjugate to a ``partial" normal form $X_\alpha+N+R$, that is a  normal form $X_\alpha+N$ where $N$ commutes with $X_\alpha$, plus a ``small" remainder $R$. In general, this remainder cannot be equal to zero, regardless of the Diophantine properties of $\alpha$. This theorem will be proved by applying finitely many steps of averaging, as usual. The novelty here is that each of this step, which in all other proofs require to estimate small divisors, will be replaced by $d$ elementary steps in which we will only deal with periodic approximations. So it could be said that our proof will reduce the general case $1\leq d\leq n$ to the special case $d=1$. In the case $d=1$, $\alpha$ is a periodic vector and its only periodic approximation is itself: the function $\Psi_{\alpha}(Q)$ is therefore constantly equals to its period $T$, and then $Q(\varepsilon)=|\alpha|(2T\varepsilon)^{-1}$.

Note that Theorem~\ref{thmfini} implies a result of ``stability" in finite time: one can compare the flow of $X_\alpha+P$ to the simpler flow of $X_\alpha+N$, during an interval of time $|t|\leq T(\varepsilon)$ where $T(\varepsilon)$ is essentially the inverse of the size of the remainder $R$. Then, one should probably be able to show that this time $T(\varepsilon)$ cannot be improved ``uniformly" in general (see \cite{Bou11}: the example there is given for $d=n$ and for Hamiltonian vector fields, but the construction should extend to our setting).

The smallness of the remainder is of course entirely tied up with the Diophantine properties of $\alpha$. If $\Delta_{\alpha}^{*}$ grows very slowly, that is if $\Psi_\alpha$ grows very fast, the size of the remainder $R$, which is always strictly smaller than $\varepsilon$, may not be much smaller. However, if $\alpha \in \mathcal{D}_d^\tau$, that is if $\alpha$ is a Diophantine vector with exponent $\tau\geq d-1$, then by Corollary~\ref{cordio}, one can find two positive constants $Q_\alpha \geq 1$ and $C_\alpha>0$ such that for all $Q\geq Q_\alpha$, $\Psi_\alpha(Q)\leq C_\alpha Q^\tau$. The above corollary is therefore an immediate consequence of Theorem~\ref{thmfini}.

\begin{corollary}\label{corfini}
Let $X=X_\alpha+P$ be as in~(\ref{H}), with $\alpha \in \mathcal{D}_d^\tau$, $\tau \geq d-1$, define
\[ E_{\alpha}=\max(1,d|\alpha|^{-1},Q_\alpha), \quad \kappa_{d,\alpha}=2^6d^2(2^d-1)|\alpha|^{-1},  \]
and assume that
\begin{equation}\label{thr00}
\epsilon < (2C_\alpha)^{-1}(E_\alpha)^{-(\tau+1)}, \quad \epsilon \leq (2C_\alpha)^{-1}(\kappa_{d,\alpha} s^{-1})^{-(\tau+1)}. 
\end{equation}
Then there exists a real-analytic embedding $\Phi : \T^n_{s/2} \rightarrow \T^n_{s}$ such that
\[ \Phi^*X=X_\alpha+N+R, \quad [X_\alpha,N]=0  \]
with the estimates $|\Phi-\mathrm{Id}|_{s/2}<d^2|\alpha|^{-1}(2C_\alpha\varepsilon)^{\frac{1}{1+\tau}}$ and
\[ |N|_{s/2}<2\varepsilon, \quad |R|_{s/2}< 2\varepsilon \exp\left(-(\ln2)\kappa_{d,\alpha}^{-1}s(2C_\alpha\varepsilon)^{-\frac{1}{1+\tau}}\right).\] 
\end{corollary}

For our second result, we have to impose two restrictions on our vector $\alpha$. First we assume that it satisfies the following integral condition:
\begin{equation}\label{condA}
\int_{\Delta_\alpha(1)}^{+\infty} \frac{dx}{x\Delta_{\alpha}^*(x)}< \infty, \tag{A}
\end{equation}
where $\Delta_{\alpha}^*$ is the function defined in~(\ref{deltastar}). Using Corollary~\ref{corbr}, we will show in appendix~\ref{app1}, Lemma~\ref{lemmeapp}, that $\alpha$ satisfies this condition~(\ref{condA}) if and only if $\alpha \in \mathcal{B}_d$, that is condition~(\ref{condA}) is equivalent to the Bruno-Rüssmann condition. Furthermore, we have to restrict to the case $d=n$: under such an assumption, any real-analytic vector field $N$ such that $[N,X_\alpha]=0$ is constant, that is $N=X_\beta$ for some $\beta \in \C^n$.

\begin{theorem}\label{thminfini}
Let $X=X_\alpha+P$ be as in~(\ref{H}), with $\alpha$ satisfying condition~(\ref{condA}) and $d=n$, define 
\[ Q(\varepsilon)=\Delta_{\alpha}^*\left((3\varepsilon)^{-1}\right), \quad r(\varepsilon)=Q(\varepsilon)^{-1}+(\ln 2)^{-1}\int_{\Delta_{\alpha}(Q(\varepsilon))}^{+\infty}\frac{dx}{x\Delta_{\alpha}^*(x)}, \quad \kappa_{d,\alpha}=2^6d^2(2^d-1)|\alpha|^{-1} \] 
and assume that
\begin{equation}\label{thr000}
Q(\varepsilon)>\max(1,d|\alpha|^{-1}), \quad r(\varepsilon) \leq \kappa_{d,\alpha}^{-1} s.
\end{equation}
Then there exist a unique constant $\beta\in\C^n$ and a real-analytic embedding $\Phi : \T^n_{s/2} \rightarrow \T^n_{s}$ such that
\[ \Phi^*(X+X_\beta)=X_\alpha \]
with the estimates
\[ |\Phi-\mathrm{Id}|_{s/2}\leq 3^{-1}d^{2}|\alpha|^{-1}r(\varepsilon), \quad |\beta|\leq 2\varepsilon.\]   
\end{theorem}

The above theorem states that by adding a ``modifying term" $X_\beta$ (in the terminology of Moser, \cite{Mos67}) to the perturbed vector field $X=X_\alpha+P$, the modified perturbed vector field $X+X_\beta=X_{\alpha+\beta}+P$ can be analytically conjugated to $X_\alpha$. In view of Theorem~\ref{thmfini} in the case $d=n$, the result says that if we modify the original perturbed vector field by adding a vector field $N$ such that $[N,X_\alpha]$, then the modified perturbed vector field can be analytically conjugate to $X_\alpha$. Hence such a statement can be called an ``inverted" normal form, and it implies a result of ``stability" in infinite times (but not for the original vector field). Theorem~\ref{thminfini} is exactly the classical KAM theorem for constant vector fields on the torus first proved by Arnold (\cite{Arn61}) and Moser (\cite{Mos66}). The proof we will give here is very close to the proof of Theorem~\ref{thmfini}, replacing a finite iteration by an infinite iteration, the latter being possible by the restrictions imposed on $\alpha$. As for the proof of Theorem~\ref{thmfini}, only the proof of the averaging step is new, the induction is then essentially classical. Also, as in \cite{Rus10} or \cite{Pos11}, the speed of convergence in the induction process is linear, and we can artificially prescribed any rate of convergence $0<c<1$ without any difficulties (but for the statement we simply chose $c=1/2$). 

Now let us briefly discuss the restrictions we imposed on $\alpha$. First an arithmetic condition is known to be necessary in the case $n=2$, and the condition we used is known to be optimal by a result of Yoccoz (see \cite{Yoc02} for instance, where the discrete version of the problem, which is the analytic linearization of circle diffeomorphisms, is considered). However, for $n\geq 3$, nothing is known, and the flexibility in the proof might suggest that condition~(\ref{condA}) in Theorem~\ref{thminfini} is perhaps not ``optimal", as trying to get rid of any artificial choice in the proof might lead to a perhaps weaker arithmetical condition (and the speed of convergence in the iteration might be sub-linear). Finally, we had to restrict also to $d=n$, and one may ask if the statement remains true for any $1\leq d\leq n$, in the sense that there exists a real-analytic vector field $N$ which commutes with $X_\alpha$ such that $X+N$ can be analytically conjugated to $X_\alpha$.

\subsection{Periodic averaging}\label{ss31}

In this paragraph, we will prove a result about periodic averaging. We consider a periodic frequency $\omega \in \R^n \setminus \{0\}$ with minimal period $T$, and let $X_\omega=\omega$. Recall that averaging along such a periodic frequency, no ``small divisors" arise, one can solve the associated homological equation by a simple integral formula without expanding in Fourier series and cutting Fourier modes. In terms of small divisors, one can notice that for all $k\in \Z^n$ such that $k\cdot\omega\neq 0$, then $|k\cdot\omega|\geq T^{-1}$, but of course we will not use this point of view.

Given $\varpi \in \R^n \setminus \{0\}$, we let $X_\varpi=\varpi$ and given two parameters $\mu>0$ and $\delta>0$, we consider a real-analytic vector field on $\T^n_s$ of the form
\begin{equation}\label{Y}
Y=X_\omega+X_\varpi+S+P, \quad |\varpi|\leq \mu, \quad |S|_{s}\leq \delta, \quad |P|_{s}\leq \varepsilon.
\end{equation}
Here $X_\omega$ is considered as unperturbed, $P$ is our original perturbation and the vector fields $X_\varpi$ and $S$ are ``parameters" which are free for the moment, but will be determined subsequently. The vector field $X_\varpi$ represents a shift of frequency, and will be chosen in \S\ref{ss32}: more precisely, we will approximate our original vector $\alpha$ by a periodic vector $\omega$, and then choose $\varpi=\alpha-\omega$ so that $X_\alpha=X_\omega+X_\varpi$. The vector field $S$ can be ignored at this stage, it will become important only for the inductions later on in \S\ref{ss33} and \S\ref{ss34}.  

\begin{lemma}\label{one}
Consider $Y$ as in~(\ref{Y}), and for $0<\varsigma<s$, $0<b<1$, assume that 
\begin{equation}\label{thr}
\varepsilon+\delta\leq \mu, \quad 2\varepsilon \leq \mu, \quad 2^4T\mu\varsigma^{-1}\leq b.
\end{equation}
Then, setting 
\[ [P]=\int_{0}^{1}P\circ X_{T\omega}^t dt, \quad V=T\int_{0}^{1}(P-[P])\circ X_{T\omega}^t tdt \] 
the map $V^1 : \T^n_{s-\varsigma} \rightarrow \T^n_{s}$ is well-defined real-analytic embedding, and
\[  (V^1)^*Y=X_\omega+X_\varpi+S+[P]+\tilde{P} \]
with the estimates 
\[ |V^1-\mathrm{Id}|_{s-\varsigma}\leq T\varepsilon, \quad |[P]|_{s}\leq\varepsilon , \quad |\tilde{P}|_{s-\varsigma}\leq b\varepsilon. \]
\end{lemma}

Note that our condition~(\ref{thr}) implies in particular that $T\varepsilon$ is sufficiently small with respect to $\varsigma$, and at this stage this is essentially the only assumption we need (let us also note that the second part of the condition~(\ref{thr}) is here just for convenience, as it will be implied by the first part). We artificially introduced other assumptions in order to prescribe already the size of the remainder $\tilde{P}$ to $b\varepsilon$, for an arbitrary $0<b < 1$. Note also that our vector field $V$, and therefore the transformation $V^1$, is independent of the choice of $X_\varpi$ and $S$. 

In the special case where $X_\varpi=S=0$, this is just an averaging along the periodic flow generated by $X_\omega$. From an analytical point of view, $X_\omega$ induces a semi-simple linear differential operator $\mathcal{L}_{X_\omega}=[\,\cdot\,,X_\omega]$ acting on the space of vector fields: $[P]$ is just the projection of $P$ onto the kernel of $\mathcal{L}_{X_\omega}$, hence $[[P],X_\omega]=0$, then $P-[P]$ lies in the range of $\mathcal{L}_{X_\omega}$ and $V$ is its unique pre-image, that is $[V,X_\omega]=P-[P]$. From a geometrical point of view, let $\mathcal{F}^1$ be the one-dimensional foliation tangent to $X_\omega$: since $\omega$ is periodic, the leaves $\mathcal{F}^1$ are compact and diffeomorphic to $\T^1$. Then $[P]$ is just the mean value of $P$ on each leaf of $\mathcal{F}^1$, the vector field $P-[P]$ has therefore zero mean so we can integrate it along each leaf of $\mathcal{F}^1$ to obtain a periodic vector field $V$. The statement that there is no ``small divisors" can be expressed by saying that the inverse operator of $\mathcal{L}_{X_\omega}$ (defined on the image of $\mathcal{L}_{X_\omega}$) is bounded by $T$, which simply comes from the fact that the leaves of $\mathcal{F}^1$ are compact. The proof of Lemma~\ref{one} is essentially classical, but for completeness we give all the details, using technical estimates which are contained in Appendix~\ref{app2}.

\begin{proof}[Proof of Lemma~\ref{one}]
Since $|P|_s\leq \varepsilon$, then obviously $|[P]|_s\leq \varepsilon$ and $|V|_s\leq T\varepsilon$. From the second and third part of~\eqref{thr}, we have in particular $T\varepsilon < \varsigma$, hence by Lemma~\ref{tech1}, the map $V^1 : \T^n_{s-\varsigma} \rightarrow \T^n_{s}$ is a well-defined real-analytic embedding and  
\[ |V^1-\mathrm{Id}|_{s-\varsigma}\leq|V|_s\leq T\varepsilon.\]
Now we can write
\begin{equation}\label{tay}
(V^1)^*Y=(V^1)^*X_\omega+(V^1)^*(X_\varpi+S+P) 
\end{equation}
and using the general equality
\[ \frac{d}{dt}(V^t)^*F=(V^t)^*[F,V] \]
for an arbitrary vector field $F$, we can apply Taylor's formula with integral remainder to the right-hand side of~\eqref{tay}, at order two for the first term and at order one for the second term, and we get
\[ (V^1)^*Y=X_\omega+[X_\omega,V]+\int_{0}^{1}(1-t)(V^t)^*[[X_\omega,V],V]dt+X_\varpi+S+P+\int_{0}^{1}(V^t)^*[X_\varpi+S+P,V]dt. \]
Now let us check that the equality $[V,X_\omega]=P-[P]$ holds true:  let us denote $G=P-[P]$ and $DV$ the differential of $V$, then since $X_\omega$ is a constant vector field, we have
\[ [V,X_\omega]=DV.\omega=T\int_{0}^{1}D(G\circ X_{T\omega}^{t}).\omega tdt=\int_{0}^{1}D(G\circ X_{T\omega}^{t}).T\omega tdt \]
so using the chain rule
\[ [V,X_\omega]=\int_{0}^{1}\frac{d}{dt}(G\circ X_{T\omega}^{t}) tdt \]
and an integration by parts
\[ [V,X_\omega]=\left.(G\circ X_{T\omega}^{t})t\right\vert_{0}^{1}-\int_{0}^{1}G\circ X_{T\omega}^{t}dt=G,  \]
where in the last equality, $G\circ X_{T\omega}^{1}=G$ since $T\omega\in\Z^n$ and the integral vanishes since $[G]=0$. So using the equality $[V,X_\omega]=P-[P]$, that can be written as $[X_\omega,V]+P=[P]$, we have 
\[ (V^1)^*Y=X_\omega+X_\varpi+S+[P]+\int_{0}^{1}(1-t)(V^t)^*[[X_\omega,V],V]dt+\int_{0}^{1}(V^t)^*[X_\varpi+S+P,V]dt,\]
and if we set
\[ P_t=tP+(1-t)[P], \quad \tilde{P}=\int_{0}^{1}(V^t)^*[P_t+X_\varpi+S,V]dt \]
and use again the equality $[X_\omega,V]=[P]-P$ we eventually obtain
\[ (V^1)^*Y=X_\omega+X_\varpi+S+[P]+\tilde{P}.\]
It remains to estimate $\tilde{P}$, and for that let us write $U=[X_\varpi+S+P,V]$. Using the second and third part of~\eqref{thr}, we have $T\varepsilon \leq 2^{-1}T\mu \leq 2^{-5}\varsigma \leq (8e)^{-1}\varsigma=(4e)^{-1}\varsigma/2$ and therefore by Lemma~\ref{tech3}, we can estimate
\begin{equation}\label{est1}
|\tilde{P}|_{s-\varsigma}\leq 2|U|_{s-\varsigma/2}. 
\end{equation}
Now the term $U$ is just a sum of three Lie brackets, and as $|P_t|_s\leq \varepsilon$, each of them can be estimated by Lemma~\ref{tech2} and we obtain
\begin{equation}\label{est2}
|U|_{s-\varsigma/2}\leq 4T\varepsilon\varsigma^{-1}(\varepsilon+\mu+\delta)\leq 8T\varsigma^{-1}\mu\varepsilon 
\end{equation}
where the last inequality uses the first part of~\eqref{thr}. Now from~\eqref{est1}, \eqref{est2} and the last part of \eqref{thr}, we finally obtain
\[ |\tilde{P}|_{s-\varsigma} \leq 16T\varsigma^{-1}\mu\varepsilon \leq b\varepsilon \]
which is the estimate we wanted. 
\end{proof}

\subsection{Quasi-periodic averaging}\label{ss32}

Now we consider our original unperturbed vector field $X_\alpha$ and a real-analytic vector field on $\T^n_s$ of the form
\begin{equation}\label{Z}
Z=X_\alpha+S+P, \quad |S|_{s}\leq \delta, \quad |P|_{s}\leq \varepsilon,
\end{equation} 
which, when $S=0$, corresponds to our vector field $X$ as in~(\ref{H}). Recall that to $\alpha\in\R^n\setminus\{0\}$ is associated a $d$-dimensional vector subspace $F_{\alpha} \subseteq \R^n$, which is the smallest rational subspace containing $\alpha$, and a $\Z$-module $\Lambda_{\alpha}=\Z^n\cap F_{\alpha}$ of rank $d$. 

For a parameter $Q>\max(1,d|\alpha|^{-1})$ to be chosen later, we will use Proposition~\ref{corutile} to find $d$ periodic vectors $\omega_j \in \R^n\setminus\{0\}$, with periods $T_j$ essentially bounded by $\Psi_\alpha(Q)$, such that each $\omega_j$ is essentially a $Q$-approximation of $\alpha$ and the integer vectors $T_j\omega_j$ form a $\Z$-basis of $\Lambda_{\alpha}$. Then applying inductively $d$ times Lemma~\ref{one}, we will obtain the following result. 

\begin{lemma}\label{multi}
Consider $Z$ as in~(\ref{Z}), and for $Q>\max(1,d|\alpha|^{-1})$, $0<\sigma<s$, and $0<c< 1$, assume that 
\begin{equation}\label{thr2}
Q\Psi_{\alpha}(Q)(\varepsilon+\delta)\leq 1, \quad 2Q\Psi_{\alpha}(Q)\varepsilon\leq 1, \quad 2^4d^2(2^d-1)|\alpha|^{-1}Q^{-1}\sigma^{-1}\leq c.
\end{equation}
Then there exists a real analytic embedding $\Phi : \T^n_{s-\sigma} \rightarrow \T^n_{s}$ such that
\[ \Phi^* Z=X_\alpha+S+\overline{P}+P^+, \quad  \overline{P}=\int_{\vartheta \in F_{\alpha}/\Lambda_{\alpha}}P\circ X_{\vartheta}^{1} d\vartheta \]
with the estimates 
\[ |\Phi-\mathrm{Id}|_{s-\sigma}\leq d^2|\alpha|^{-1}\Psi_{\alpha}(Q)\varepsilon, \quad |\overline{P}|_{s}\leq \varepsilon , \quad |P^+|_{s-\sigma}\leq c\varepsilon. \] 
\end{lemma}

Let $\tilde{\mathcal{F}}_{\alpha}^d$ be the $d$-dimensional linear foliation of $\R^n$ given by $\R^n=\bigsqcup_{a\in F_{\alpha}^\perp}F_{\alpha}^a$, where $F_{\alpha}^a=a+F_{\alpha}$ is the translate of $F_{\alpha}$ by $a \in F_{\alpha}^\perp$. If $\pi : \R^n \rightarrow \T^n$ is the canonical projection, then $\pi(F_{\alpha})=F_{\alpha}/\Lambda_{\alpha}$ is compact and diffeomorphic to $\T^d$, hence $\mathcal{F}_{\alpha}^d=\pi(\tilde{\mathcal{F}}_{\alpha}^d)$ is a $d$-dimensional foliation of $\T^n$ with compact leaves diffeomorphic to $\T^d$. So in the above statement, $d\vartheta$ is the Haar measure on the compact quotient group $F_{\alpha}/\Lambda_{\alpha}$ and $X_{\vartheta}^{1}$ is the time-one map of the constant vector field $X_{\vartheta}=\vartheta$. Note that by Birkhoff's ergodic theorem, one can also express $\overline{P}$ in the more classical format
\[ \overline{P}=\lim_{s\rightarrow +\infty}\frac{1}{s}\int_{0}^{s}P\circ X_{\alpha}^{t}dt \]
since each leaf of $\mathcal{F}_{\alpha}^d$ is invariant by $X_\alpha$, and the restriction of $X_\alpha$ to each such leaf is uniquely ergodic (hence the above convergence is uniform). The vector field $\overline{P}$ can be equivalently considered as being the projection of $P$ onto the kernel of the operator $\mathcal{L}_{X_\alpha}=[\,\cdot\,,X_\alpha]$ or the mean value of $P$ along the leaves of $\mathcal{F}_{\alpha}^d$. As before, we will notice that the transformation $\Phi$ constructed in the above lemma is independent of the choice of $S$.

In the special case where $S=0$, this is just an averaging along the quasi-periodic flow generated by $X_\alpha$. In the classical approach, one tries to solve the equation $[W,X_\alpha]=P-\overline{P}$, or equivalently to integrate $P-\overline{P}$ along the leaves of the one-dimensional foliation tangent to $X_\alpha$, to find a coordinate transformation $\Phi=W^1$. But then small divisors inevitably arise as the inverse operator of $\mathcal{L}_{X_\alpha}$ is now unbounded, since the leaves of the foliation defined by $X_{\alpha}$ are non-compact. To overcome this difficulty, usually one introduce a parameter $K\geq 1$ and replace $P$ by a polynomial approximation $P_K$ obtained by cutting higher Fourier modes of $P$. The term $P-P_K$ is then considered as an ``error" and thus one solves only an ``approximate" equation. 

Here we shall use a significantly different though ``dual" approach, replacing the quasi-periodic flow $X_\alpha$ by $d$ independent approximating periodic flows $X_{\omega_1}, \dots, X_{\omega_d}$, which at each point will be tangent to the foliation $\mathcal{F}_{\alpha}^d$. More precisely, for a given $Q$ sufficiently large, we will approximate $\alpha$ by $d$ independent $T_j$-periodic vectors $\omega_j$ such that the integer vectors $T_j\omega_j$ form a $\Z$-basis of $\Lambda_\alpha$. Replacing $\alpha$ by $\omega_j$, the terms $\alpha-\omega_j$ will be considered as ``errors" and we will solve $d$ equations $[V_j,X_{\omega_j}]=P_{j-1}-[P_{j-1}]_j$ successively, starting with $P_0=P$ and choosing at each step $P_j=[P_{j-1}]_j$, where $[\,\cdot\,]_j$ denotes the average along the periodic flow of $\omega_j$. At each step, only the inverse operator of $\mathcal{L}_{X_{\omega_j}}=[\,\cdot\,,X_{\omega_j}]$ will be involved, and the latter is bounded by $T_j$. At the end, we will find a coordinate transformation $\Phi$ as the composition of time-one maps $V_1^1\circ \cdots \circ V_d^1$, even though we do not integrate along the flow of $X_\alpha$ (that is we do not find a generating function of $\Phi$). Also, $P_d$ will be equal to the (space) average of $P$ along the leaves of the foliation spanned by $X_{\omega_1}, \dots, X_{\omega_d}$ (which is the foliation $\mathcal{F}_{\alpha}^d$) and therefore $P_d$ is also the (time) average along the flow of $X_\alpha$.  Note that the role of our ``approximating" parameter $Q\geq 1$ is dual to the use of the parameter $K\geq 1$ in the classical approach, and by solving exactly the $d$ equations we mentioned above we will essentially solves the classical approximate equation. 

\begin{proof}[Proof of Lemma~\ref{multi}]
For a given $Q>\max(1,d|\alpha|^{-1})$, we apply Proposition~\ref{corutile}: there exists $d$ periodic vectors $\omega_1, \dots, \omega_d$, of periods $T_1, \dots, T_d$, such that $T_1\omega_1, \dots, T_d\omega_d$ form a $\Z$-basis of $\Lambda_\alpha$ and for $j\in\{1,\dots,d\}$,
\[ |\alpha-\omega_j|\leq d(|\alpha|T_j Q)^{-1}, \quad |\alpha|^{-1} \leq T_j \leq |\alpha|^{-1}d\Psi_d(Q).\]
Now we define $\varpi_j=\alpha-\omega_j$ and $\mu_j=d(|\alpha|T_j Q)^{-1}$ so that $|\varpi_j|\leq \mu_j$, and we set $\varsigma=d^{-1}\sigma$ and $b=(2^d-1)^{-1}c$. Then $0<\sigma<s$ and $0<c< 1$ implies $0<\varsigma<s$ and $0<b< 1$, and for any $j\in\{1,\dots,d\}$, by~(\ref{thr2}), the conditions
\begin{equation}\label{thri}
\varepsilon+\delta\leq \mu_j, \quad 2\varepsilon\leq \mu_j, \quad 2^4T_j\mu_j\varsigma^{-1}\leq b,
\end{equation}
are satisfied. Set $P_0=P$ and define inductively 
\[ P_j=\int_{0}^{1}P_{j-1}\circ X_{T_j\omega_{j}}^tdt=[P_{j-1}]_j, \quad j\in\{1,\dots,d\}.\]
Obviously, $|P_j|_{s}<\varepsilon$ and
\begin{equation*}
P_d = \int_{0}^{1}\dots\int_{0}^{1}P\circ X_{T_1\omega_{1}}^{t_1}\circ \cdots \circ X_{T_d\omega_{d}}^{t_d} dt_1\dots dt_d=\overline{P} 
\end{equation*}
since 
\[ \{t_1T_1\omega_1+\cdots+t_dT_d\omega_d \; | \;  0\leq t_j <1, \; j\in\{1,\dots,d\} \} \subseteq F_\alpha \] 
is a fundamental domain of $F_{\alpha}/\Lambda_{\alpha}$. Now for $j\in\{0,\dots,d\}$, set $s_j=s-j\varsigma$, then $s_0=s$ and $s_d=s-\sigma$. Then we claim that for any $j\in\{0,\dots,d\}$, there exist a real analytic embedding $\Phi_j : \T^n_{s_j} \rightarrow \T^n_{s}$ such that
\[ (\Phi_j^*)Z=X_\alpha+S+P_j+P_j^+ \]
with the estimates 
\[ |\Phi_j-\mathrm{Id}|_{s_j}\leq dj|\alpha|^{-1}\Psi_{\alpha}(Q)\varepsilon, \quad |P_j^+|_{s_j}\leq (2^j-1)b\varepsilon. \] 
For $j=0$, since $P_0=P$, letting $\Phi_0$ be the identity and $P_0^+=0$, there is nothing to prove. So assume the statement is true for some $j\in\{0,\dots,d-1\}$, and let us prove it remains true for $j+1$. To do so, let us write, $X_\alpha=X_{\omega_j}+X_{\varpi_j}$ so that by~\eqref{thri}, Lemma~\ref{one} can be applied to the vector field
\[ Y_j=(\Phi_j^*)Z-P_j^+=X_\alpha+S+P_j=X_{\omega_j}+X_{\varpi_j}+S+P_j \]
and if 
\[ V_{j+1}=T_{j+1}\int_{0}^{1}(P_j-P_{j+1})\circ X_{T_{j+1}\omega_{j+1}}^t tdt,  \] 
then $V_{j+1}^{1} : \T^{n}_{s_{j+1}} \rightarrow \T^{n}_{s_{j}}$ is a well-defined analytic embedding for which
\[ (V_{j+1}^{1})^*Y_j=X_{\omega_j}+X_{\varpi_j}+S+P_{j+1}+\tilde{P}_{j+1}=X_{\alpha}+S+P_{j+1}+\tilde{P}_{j+1}, \]
with the estimate 
\[ |V_{j+1}^{1}-\mathrm{Id}|_{s_{j+1}}\leq T_{j+1} \varepsilon\leq d|\alpha|^{-1}\Psi_{\alpha}(Q)\varepsilon, \quad |\tilde{P}_{j+1}|_{s_{j+1}}\leq b\varepsilon. \] 
So we define $\Phi_{j+1}=\Phi_j\circ V_{j+1}^{1}$ and $P_{j+1}^+=\tilde{P}_{j+1}+(V_{j+1}^{1})^*P_j^+$ hence
\[ (\Phi_{j+1}^*)Z=X_\alpha+S+P_{j+1}+P_{j+1}^+. \]
For the estimates, we write $\Phi_{j+1}-\mathrm{Id}=\Phi_j\circ V_{j+1}^{1}-V_{j+1}^{1}+V_{j+1}^{1}-\mathrm{Id}$ so that
\[ |\Phi_{j+1}-\mathrm{Id}|_{s_{j+1}}\leq |\Phi_j-\mathrm{Id}|_{s_j}+|V_{j+1}^{1}-\mathrm{Id}|_{s_{j+1}}\leq d(j+1)|\alpha|^{-1}\Psi_{\alpha}(Q)\varepsilon  \]
and using Lemma~\ref{tech3}, we have
\[ |(V_{j+1}^{1})^*P_j^+|_{s_{j+1}}\leq 2|P_j^+|_{s_j}\leq 2(2^j-1)b\varepsilon \]
hence
\[ |P_{j+1}^+|_{s_{j+1}}\leq b\varepsilon + 2(2^j-1)b\varepsilon=(1+2(2^j-1))b\varepsilon=(2^{j+1}-1)b\varepsilon. \]
So this proves the claim, and setting $\Phi=\Phi_d$ and $P^+=P_d^+$, as $P_d=\overline{P}$ and $b=(2^d-1)^{-1}c$, this also proves the lemma.
\end{proof}

\subsection{Proof of Theorem~\ref{thmfini}}\label{ss33}

We are finally ready to give the proof of Theorem~\ref{thmfini}, which consists in applying inductively Lemma~\ref{multi} as many times as we can. For simplicity, this lemma will be applied with $c=1/2$. 

\begin{proof}[Proof of Theorem~\ref{thmfini}]
Let $C_{d,\alpha}=2^5d^2(2^d-1)|\alpha|^{-1}$ and $\kappa_{d,\alpha}=2C_{d,\alpha}$. For $m\geq 1$ to be chosen below, and for $i\in\{0,\dots,m\}$, define 
\[ \varepsilon_{i}=2^{-i}\varepsilon, \quad \gamma_i=(1-2^{-i})2\varepsilon=\sum_{k=0}^{i-1}\varepsilon_i, \quad \sigma_i=(2m)^{-1}s. \]
Note that $\varepsilon_i+\gamma_i=\gamma_{i+1}<2\varepsilon$ for $i\in\{0,\dots,m\}$. Let $s_{-1}=s_0=s$ and define inductively $s_{i}=s_{i-1}-\sigma_i$ for $i\in\{1,\dots,m\}$ so that $s_m=s/2$. For $Q>\max(1,d|\alpha|^{-1})$ to be chosen below, if the conditions
\begin{equation}\label{thrii}
Q\Psi_{\alpha}(Q)(\varepsilon_i+\gamma_i)\leq 1, \quad 2Q\Psi_{\alpha}(Q)\varepsilon_i\leq 1, \quad C_{d,\alpha}Q^{-1}\sigma_{i}^{-1}\leq 1
\end{equation}
are satisfied, then we claim that for all $i\in\{0,\dots,m\}$, there exists a real-analytic embedding $\Phi^i : \T^n_{s_{i}} \rightarrow \T^n_{s}$ such that
\[ (\Phi^i)^*X=X_\alpha+N_i+P_i, \quad [N_i,X_\alpha]=0, \]
with the estimates
\[ |\Phi^i-\mathrm{Id}|_{s_{i}}\leq d^2|\alpha|^{-1}\Psi_{\alpha}(Q)\gamma_i, \quad |N_{i}|_{s_{i-1}} \leq \gamma_{i}, \quad |P_{i}|_{s_{i}}\leq \varepsilon_{i}. \]
Indeed, for $i=0$, choosing $\Phi^0$ to be the identity, $N_0=0$ and $P_0=P$, there is nothing to prove. So assume the statement holds true for some $i\in\{0,\dots,m-1\}$, and let us prove it remains true for $i+1$. By the induction hypothesis and~(\ref{thrii}), we can apply Lemma~\ref{multi} with $c=1/2$ to $Z=(\Phi^i)^*X=X_\alpha+N_i+P_i$, so with $S=N_i$ and $\delta=\gamma_i$, and we find a real-analytic embedding $\Phi_{i+1} : \T^n_{s_{i+1}} \rightarrow \T^n_{s_{i}}$ such that
\[ \Phi_{i+1}^{*}Z=X_\alpha+N_i+\overline{P_i}+P_i^+ \]
with the estimates
\[ |\Phi_{i+1}-\mathrm{Id}|_{s_{i+1}}\leq d^2|\alpha|^{-1}\Psi_{\alpha}(Q)\varepsilon_i, \quad |\overline{P_i}|_{s_{i}}\leq \varepsilon_i, \quad |P_i^+|_{s_{i+1}}\leq 2^{-1}\varepsilon_i=\varepsilon_{i+1}. \]
Since $\varepsilon_i+\gamma_i=\gamma_{i+1}$ and $[\overline{P_i},X_\alpha]=0$, setting $\Phi^{i+1}=\Phi^i\circ\Phi_{i+1}$, $N_{i+1}=N_i+\overline{P_i}$ and $P_{i+1}=P_i^+$, the statement obviously holds true for $i+1$ and this proves the claim. Therefore setting $\Phi=\Phi^m : \T^n_{s/2} \rightarrow \T^n_{s}$, $N=N_m$ and $R=P_m$, we obtain
\[ \Phi^*X=X_\alpha+N+R, \quad [X_\alpha,N]=0, \]
with the estimates
\[ |\Phi-\mathrm{Id}|_{s/2}\leq d^2|\alpha|^{-1}\Psi_{\alpha}(Q)\gamma_m<2d^2|\alpha|^{-1}\Psi_{\alpha}(Q)\varepsilon,\]
and
\[ |N|_{s/2}\leq |N|_{s_{m-1}}\leq \gamma_m<2\varepsilon, \quad |R|_{s/2}\leq \varepsilon_m=2^{-m}\varepsilon. \] 
Now it remains to choose our parameters $Q\geq 1$ and $m\geq 1$ in order to fulfil our assumption~(\ref{thrii}). First, $\varepsilon_i+\gamma_i< 2\varepsilon$ and $2\varepsilon_i\leq 2\varepsilon$ for all $i\in\{0,\dots,m\}$, hence the first and second part of~(\ref{thrii}) are satisfied if $2Q\Psi_{\alpha}(Q)\varepsilon\leq 1$, and the latter is satisfied if $\Delta_{\alpha}(Q)=Q\Psi_{\alpha}(Q)\leq (2\varepsilon)^{-1}$, that is if $Q=\Delta_{\alpha}^*\left((2\varepsilon)^{-1}\right)$. As for the third part of~(\ref{thrii}), since $\sigma_i^{-1}=2ms^{-1}$ the latter is satisfied if we choose $m$ to be the largest integer smaller than $(2C_{d,\alpha})^{-1}sQ=\kappa_{d,\alpha}^{-1}sQ$. So eventually $Q$ and $m$ are chosen as follows:
\[ Q=\Delta_{\alpha}^*\left((2\varepsilon)^{-1}\right), \quad m=\left\lfloor(\kappa_{d,\alpha})^{-1} sQ\right\rfloor=\left\lfloor(\kappa_{d,\alpha})^{-1}s\Delta_{\alpha}^*\left((2\varepsilon)^{-1}\right)\right\rfloor.  \]
Our threshold~(\ref{thr0}) in the statement of Theorem~\ref{thmfini} ensures that both $Q>\max(1,d|\alpha|^{-1})$ and $m\geq 1$, and we have obtained
\[ \Phi^*X=X_\alpha+N+R, \quad [X_\alpha,N]=0, \]
with the estimates
\[ |\Phi-\mathrm{Id}|_{s/2}< 2d^2|\alpha|^{-1}\Psi_{\alpha}(Q)\varepsilon \leq d^2|\alpha|^{-1}Q^{-1}\]
and since $m>\kappa_{d,\alpha}^{-1}sQ-1$, 
\[ |N|_{s/2}<2\varepsilon, \quad |R|_{s/2}< 2\varepsilon 2^{-\kappa_{d,\alpha}^{-1}sQ}=2\varepsilon \exp\left(-\ln2(\kappa_{d,\alpha})^{-1}sQ\right). \]
This was the statement to prove.
\end{proof}

\subsection{Proof of Theorem~\ref{thminfini}}\label{ss34}

First let us point out that even if we impose an arithmetical condition on $\alpha$ and we restrict to $d=n$, it seems very difficult to use the scheme of the proof of Theorem~\ref{thmfini} to go from a finite iteration to an infinite iteration, that is to analytically conjugate $X_\alpha+P$ to a normal form $X_\alpha+N$, with $[X_\alpha,N]=0$. Of course, there is no problem at the formal level as one can construct a formal transformation $\Phi^{\infty}$ and a formal vector field $N_{\infty}$ such that formally $(\Phi^{\infty})^*X=X_\alpha+N_{\infty}$. However, because of the presence of the formal vector field $N_{\infty}$, the formal transformation $\Phi^\infty$ cannot converge in general. A technical explanation in our situation goes as follows. To make infinitely many iterations in the previous scheme, instead of fixing a large $Q$ with respect to a small $\varepsilon$, one should consider an increasing sequence of $Q_m$ tending to infinity with respect to the decreasing sequence $\varepsilon_m$ tending to zero. But then at each step we would have to apply Lemma~\ref{multi} with $\delta=\gamma_m$ and the first part of condition~(\ref{thrii}) would read $Q_m\Psi_{\alpha}(Q_m)(\varepsilon_m+\gamma_m)\leq 1$. This condition cannot hold for all $m\in\N$: $Q_m\Psi_{\alpha}(Q_m)$ has to tend to infinity, while $\varepsilon_m+\gamma_m$ does not converge to zero (it is strictly bigger than $\varepsilon$, and actually converges to $2\varepsilon$). 

Now we will explain how a slight modification of this scheme (and therefore of the expected result) will lead us to the proof of Theorem~\ref{thminfini}. First, for $d=n$, a vector field $N$ which commutes with $X_\alpha$ has to be constant, that is $N=X_\beta$ for some $\beta\in \C^n$. Then the idea, which goes back to Arnold and which has been largely exploited by Moser (see \cite{Mos67}), is that instead of trying to conjugate $X=X_\alpha+P$ to a normal form $X_\alpha+X_\beta$, we will try to conjugate a modified vector field $X+X_\beta=X_\alpha+X_\beta+P$ to the constant vector field $X_\alpha$. This simple change of point of view will allow us to apply the averaging procedure described in Lemma~\ref{multi} infinitely many times and to prove the convergence easily. Technically speaking, at each step we will be able to apply Lemma~\ref{multi} with $\delta=2\varepsilon_m$ (and not with $\delta=\gamma_m$ as we did before), and the condition $Q_m\Psi_{\alpha}(Q_m)(\varepsilon_m+2\varepsilon_m)\leq 1$ can and will be fulfilled for all $m\in\N$. The last part of condition~(\ref{thrii}) will be used as a definition of $\sigma_m$, namely $\sigma_m$ will be essentially equal to $Q_m^{-1}$, and the arithmetic condition on $\alpha$ will ensure that the series of $Q_m^{-1}$, and therefore the series of $\sigma_m$, is finite and can be made as small as we wishes provided we choose $Q=Q_0$ sufficiently large. 

The only real difference with the scheme of the proof of Theorem~\ref{thmfini}, where everything was explicit, is that here the modifying vector field $X_\beta$ cannot be constructed explicitly, its existence is obtained by an implicit argument (exactly like a fixed point whose existence is obtained by Picard or Newton iterations). The modifying vector field $X_\beta$ cannot be determined in advance, at each step of the iteration it is only known at a certain precision which increases with the number of steps, so that it is only at the end of iteration that $\beta$ can be uniquely determined. To make things more precise, for a given $r\geq 0$, let
\[ B_r(\alpha)=\{ x \in \C^n \; | \; |\alpha-x|\leq r \}. \]
Then Lemma~\ref{multi} can be recast as follows.

\begin{lemma}\label{multi2}
Let $X$ be as in \eqref{H} with $d=n$, $Q>\max(1,d|\alpha|^{-1})$, $0<\sigma<s$ and assume that 
\begin{equation}\label{thr22}
3Q\Psi_{\alpha}(Q)\varepsilon\leq 1, \quad 2^5d^2(2^d-1)|\alpha|^{-1}Q^{-1}\sigma^{-1}\leq 1.
\end{equation}
Then there exist an embedding $\varphi : B_{\varepsilon}(\alpha)\rightarrow B_{2\varepsilon}(\alpha)$ and a real analytic embedding $\Phi : \T^n_{s-\sigma} \rightarrow \T^n_{s}$ such that for all $x\in B_{\varepsilon}(\alpha)$,
\[ \Phi^* (X_{\varphi(x)}+P)=X_{x}+P^+ \]
with the estimates 
\[ |\Phi-\mathrm{Id}|_{s-\sigma}\leq d^2|\alpha|^{-1}\Psi_{\alpha}(Q)\varepsilon, \quad |P^+|_{s-\sigma}\leq 2^{-1}\varepsilon. \] 
\end{lemma}

\begin{proof}
Since $d=n$, we have $\overline{P}=X_\eta$ for some vector $\eta\in \C^n$, and as $|P|_s\leq \varepsilon$, then $|\eta|\leq \varepsilon$. So let us define $\varphi : B_{\varepsilon}(\alpha)\rightarrow B_{2\varepsilon}(\alpha)$ to be the translation 
\[ \varphi(x)=x-\eta, \quad x\in B_{\varepsilon}(\alpha).\]
Take any $x\in B_{\varepsilon}(\alpha)$, then $|\varphi(x)-\alpha|\leq |x-\alpha|+|\eta| \leq 2\varepsilon$ and we can write
\[ X_{\varphi(x)}+P=X_{\alpha}+X_{\varphi(x)-\alpha}+P \]
and by condition~\eqref{thr22}, we can apply Lemma~\ref{multi} with $S=X_{\varphi(x)-\alpha}$, $\delta=2\varepsilon$ and $c=1/2$, to find a real analytic embedding $\Phi : \T^n_{s-\sigma} \rightarrow \T^n_{s}$ such that 
\[ \Phi^*(X_{\varphi(x)}+P)=X_{\varphi(x)+\eta}+P^+=X_{x}+P^+ \]
with the estimates 
\[ |\Phi-\mathrm{Id}|_{s-\sigma}\leq d^2|\alpha|^{-1}\Psi_{\alpha}(Q)\varepsilon, \quad |P^+|_{s-\sigma}\leq 2^{-1}\varepsilon. \] 
This was the statement to prove.
\end{proof}

Let us point out that the transformation $\Phi$ in the above lemma is indeed independent of the choice of $y\in B_{2\varepsilon}(\alpha)$, simply because in Lemma~\ref{multi} the transformation is independent of the choice of $S$. To understand this more precisely, let us notice that the transformation in Lemma~\ref{multi} does not depend directly on $\alpha$, but on the choice (for a given $Q$ sufficiently large) of $n$ periodic vectors $\omega_j$ for which
\[ |\alpha-\omega_j|\leq d(|\alpha|T_j Q)^{-1}, \quad |\alpha|^{-1}\leq T_j \leq d|\alpha|^{-1}\Psi_{\alpha}(Q). \]
Now the point is that if $y$ is sufficiently close to $\alpha$, one can choose the same periodic vectors and consequently the same transformation: the first part of condition~\eqref{thr22} implies (in particular) that 
\[ |y-\alpha|\leq 2\varepsilon \leq (Q\Psi_\alpha(Q))^{-1} \leq d(|\alpha|T_j Q)^{-1}  \]
so that for all $1\leq i \leq n$,
\[ |y-\omega_i|\leq |y-\alpha|+|\alpha-\omega_i|\leq 2d(|\alpha|T_j Q)^{-1}. \]
Hence the same transformation can be used if we replace $\alpha$ by $y$, and the same estimates can be obtained (by imposing slightly stronger assumptions, as it is done in the first part of condition~\eqref{thr22}, to compensate the factor $2$ in the last inequality).   

Now it will be easy to use Lemma~\ref{multi2} infinitely many times to prove Theorem~\ref{thminfini}. At the beginning, $y$ can lie anywhere in $B_{2\varepsilon}(\alpha)$. After one step, if $\overline{P}=X_\eta$, then $y$ has to belong to the image of the map $\varphi$ defined above, and this image is a smaller ball, namely $B_{\varepsilon}(-\eta)$. After $m$ steps, $y$ has to belong to a ball of radius $\varepsilon_{m-1}=2^{m-1}\varepsilon$ and so at the limit, $y=\alpha+\beta$ for some unique $\beta \in \C^n$ with $|\beta|\leq 2\varepsilon$.

\begin{proof}[Proof of Theorem~\ref{thminfini}]
Let $C_{d,\alpha}=2^5d^2|\alpha|^{-1}(2^d-1)$ and $\kappa_{d,\alpha}=2C_{d,\alpha}$. For any $m\in\N$, and for $Q>\max(1,d|\alpha|^{-1})$ to be chosen below, let us define the monotonic sequences
\[ \varepsilon_m=2^{-m}\varepsilon, \quad \Delta_m=2^{m}\Delta_{\alpha}(Q), \quad Q_m=\Delta_{\alpha}^*(\Delta_m). \]
Let also $\varepsilon_{-1}=2\varepsilon$, $s_0=s$ and for $m\in\N$, define 
\[  \sigma_m=C_{d,\alpha}Q_m^{-1}, \quad s_{m+1}=s_m-\sigma_m. \]
With these choices, since $\Delta_m\varepsilon_m=\Delta_0\varepsilon_0$ for all $m\in\N$, the first part of the condition
\begin{equation}\label{thrim}
3\Delta_m\varepsilon_m\leq 1, \quad C_{d,\alpha}Q_m^{-1}\sigma_{m}^{-1}\leq 1,
\end{equation}
is satisfied for any $m\in \N$ provided $3\Delta_0\varepsilon_0=3\Delta_{\alpha}(Q)\varepsilon\leq 1$, while the second part is always satisfied by definition of $\sigma_m$. Hence we choose $Q=\Delta_{\alpha}^*\left((3\varepsilon)^{-1}\right)$, and the requirement that $Q>\max(1,d|\alpha|^{-1})$ is ensured by the first part of our threshold~(\ref{thr000}) in the statement of Theorem~\ref{thminfini}. Now note that
\[ Q_m\Psi_{\alpha}(Q_m)=\Delta_{\alpha}(Q_m)=\Delta_{\alpha}(\Delta_{\alpha}^*(\Delta_m))\leq \Delta_m, \]
hence the condition~(\ref{thrim}) implies
\begin{equation}\label{thrimm}
3Q_m\Psi_{\alpha}(Q_m)\varepsilon_m\leq 1, \quad C_{d,\alpha}Q_m^{-1}\sigma_{m}^{-1}\leq 1.
\end{equation}
We claim that for any $m\in\N$, there exist an embedding $\varphi^m : B_{\varepsilon_{m-1}}(\alpha)\rightarrow B_{2\varepsilon}(\alpha)$ and a real analytic embedding $\Phi^m : \T^n_{s_{m}} \rightarrow \T^n_{s}$ such that for all $x_m\in B_{\varepsilon_{m-1}}(\alpha)$,
\[ (\Phi^m)^* (X_{\varphi^m(x_m)}+P)=X_{x_m}+P_m \]
with the estimates 
\begin{equation}\label{estim}
|\Phi^m-\mathrm{Id}|_{s_m}\leq 3^{-1}d^2|\alpha|^{-1}\sum_{i=0}^{m-1}Q_i^{-1}, \quad |P_{m}|_{s_{m}}\leq \varepsilon_{m}. 
\end{equation}
Indeed, for $m=0$, choosing $\varphi^0$ and $\Phi^0$ to be the identity, and $P_0=P$, there is nothing to prove. If we assume that the statement holds true for some $m\in\N$, then by \eqref{thrimm} we can apply Lemma~\ref{multi2} to the resulting vector field and an embedding $\varphi_{m+1} : B_{\varepsilon_{m}}(\alpha)\rightarrow B_{\varepsilon_{m-1}}(\alpha)$ and a real analytic embedding $\Phi_{m+1} : \T^n_{s_{m+1}} \rightarrow \T^n_{s_m}$ are constructed. It is then sufficient to let $\varphi^{m+1}=\varphi^m\circ \varphi_{m+1}$, $\Phi^{m+1}=\Phi^m\circ \Phi_{m+1}$ and $P_{m+1}=P_m^+$. The estimate $|P_{m+1}|_{s_{m+1}}\leq \varepsilon_{m+1}$ is obvious, and since $3Q_m\Psi_{\alpha}(Q_m)\varepsilon_m\leq 1$, we have 
\[|\Phi_{m+1}-\mathrm{Id}|_{s_{m+1}}\leq d^2|\alpha|^{-1}\Psi_{\alpha}(Q_m)\varepsilon_m\leq 3^{-1}d^2|\alpha|^{-1}Q_m^{-1}  \]
and the estimate for $\Phi^{m+1}$ follows.

Now let us prove that
\begin{equation}\label{limit}
\lim_{m \rightarrow +\infty}\varepsilon_m=0, \quad \lim_{m \rightarrow +\infty}s_m\geq s/2. 
\end{equation}
The first assertion is obvious. For the second one, $Q_m=\Delta_{\alpha}^*(\Delta_m)=\Delta_{\alpha}^*\left(2^{m}\Delta_{\alpha}(Q)\right)$ hence
\[ \sum_{m\geq 1}Q_m^{-1}\leq \int_{0}^{+\infty}\frac{dy}{\Delta_{\alpha}^*\left(2^{y}\Delta_{\alpha}(Q)\right)}=(\ln 2)^{-1}\int_{\Delta_{\alpha}(Q)}^{+\infty}\frac{dx}{x\Delta_{\alpha}^*(x)} \]
and therefore, using the second part of our condition~(\ref{thr000}) in the statement of Theorem~\ref{thminfini} and the fact that $Q_0=Q$, we obtain
\begin{equation}\label{limit2}
\sum_{m\geq 0}Q_m^{-1}\leq Q^{-1}+ (\ln 2)^{-1}\int_{\Delta_{\alpha}(Q)}^{+\infty}\frac{dx}{x\Delta_{\alpha}^*(x)} \leq \kappa_{d,\alpha}^{-1}s 
\end{equation}
and by the definition of $\sigma_m$, this gives
\[ \sum_{m\geq 0}\sigma_m=C_{d,\alpha}\sum_{m\geq 0}Q_m^{-1}\leq C_{d,\alpha}\kappa_{d,\alpha}^{-1}s=s/2. \]
This implies that
\[ \lim_{m \rightarrow +\infty}s_m=s-\sum_{m\geq 0}\sigma_m \geq s/2. \]
From the first part of~\eqref{limit}, when $m$ goes to infinity, $x_m$ converges to $\alpha$ and therefore $\varphi^m$ converges to a trivial map $\varphi : \{\alpha\} \rightarrow B_{2\epsilon}(\alpha)$. From~\eqref{estim} and~\eqref{limit}, $P_m$ converges to zero uniformly on every compact subsets of $\T^n_{s/2}$, while $\Phi^m$ converges to a embedding $\Phi : \T^n_{s/2} \rightarrow \T^n_s$, uniformly on every compact subsets of $\T^n_{s/2}$. Since the space of real-analytic functions is closed for the topology of uniform convergence on compact subsets, $\Phi$ is real-analytic, and from~\eqref{limit2}, we obtain the estimate
\[ |\Phi-\mathrm{Id}|_{s/2}\leq 3^{-1}d^{2}|\alpha|^{-1}\sum_{m\geq 0}Q_m^{-1} \leq 3^{-1}d^{2}|\alpha|^{-1}\left(Q^{-1}+(\ln 2)^{-1}\int_{\Delta_{\alpha}(Q)}^{+\infty}\frac{dx}{x\Delta_{\alpha}^*(x)}\right).\]
Finally, if $\beta \in \C^n$ is the unique vector such that $\varphi(\alpha)=\alpha+\beta$, then $|\beta|\leq 2\varepsilon$ and we have 
\[  \Phi^*(X+X_\beta)=X_\alpha. \]
This was the statement to prove. 
\end{proof}

\appendix

\section{Bruno-Rüssmann condition}\label{app1}

Recall that the set of Bruno-Rüssmann vectors $\mathcal{B}_d$ was defined in \S\ref{subseccordio}, and in \S\ref{sec3} we introduced another condition~(\ref{condA}). The aim of this short appendix is to prove the following proposition.

\begin{lemma}\label{lemmeapp}
A vector $\alpha$ satisfies condition~(\ref{condA}) if and only if it belongs to $\mathcal{B}_d$. 
\end{lemma}

In the proof, we shall make use of integration by parts and change of variables formulas for Stieltjes integral. 

\begin{proof}
First let us prove that there exists a continuous, non-decreasing and unbounded function $\Phi : [1,+\infty[ \rightarrow [1,+\infty[$ such that 
\begin{equation}\label{comp}
\Psi_\alpha(Q) \leq \Phi(Q) \leq \Psi_\alpha(Q+1), \quad Q\geq 1. 
\end{equation}
Recall from Proposition~\ref{proppsietpsiprime} that $\Psi_\alpha=\Psi_d$ is left-continuous and constant on each interval $]Q_l,Q_{l+1}]$, $l\in\N^*$ and on $[1,Q_1]$ with $Q_1 \geq 1$. For any $l\in\N^*$, let us choose a point $b_l\in]Q_l,Q_{l+1}[ \cap [Q_{l+1}-1,Q_{l+1}[$ and  define
\begin{equation*}
\Phi_l(Q)=
\begin{cases}
\Psi_\alpha(Q_{l+1}), \quad Q\in]Q_l,b_l],\\
\Psi_\alpha(Q_{l+1})+(Q-b_l)(Q_{l+1}-b_l)^{-1}\left(\Psi_\alpha(Q_{l+2})-\Psi_\alpha(Q_{l+1})\right), \quad Q\in]b_l,Q_{l+1}].
\end{cases}
\end{equation*} 
If $Q_1>1$, we choose a point $b_0\in [1,Q_{1}[ \cap [Q_{1}-1,Q_{1}[$ and define $\Phi_0$ on $[1,Q_1]$ similarly, otherwise if $Q_1=1$ we simply define $\Phi_0(1)=\Psi_{\alpha}(1)$, so that finally the function $\Phi$ defined by 
\[ \Phi=\Phi_0 \mathbf{1}_{[1,Q_{1}]}+\sum_{l\in\N^*}\Phi_l \mathbf{1}_{]Q_l,Q_{l+1}]}\] 
has all the wanted properties. The existence of such a function $\Phi$ shows that in Corollary \ref{corbr}, the integral condition may be written equivalently in terms of $\Psi_{\alpha}$, or also in terms of this function $\Phi$. 

Now let $\Delta_\Phi(Q)=Q\Phi(Q)$ and $\Delta_\Phi^{-1}$ denotes the functional inverse of $\Delta_\Phi$. From~\eqref{comp}, we deduce the following inequalities:
\begin{equation*}\
\Delta_\alpha(Q) \leq \Delta_\Phi(Q) \leq \Delta_\alpha(Q+1), \quad Q\geq 1. 
\end{equation*}
\begin{equation*}
\Delta^*_\alpha(x)-1 \leq \Delta^{-1}_\Phi(x) \leq \Delta^*_\alpha(x), \quad x\geq \Delta_{\alpha}(1). 
\end{equation*}
\begin{equation}\label{comp4}
\frac{1}{\Delta^*_\alpha(x)} \leq \frac{1}{\Delta^{-1}_\Phi(x)} \leq \frac{1}{\Delta^*_\alpha(x)-1}, \quad x> \Delta_{\alpha}(1). 
\end{equation}
From~\eqref{comp4} it follows that $\alpha$ satisfies condition~\eqref{condA} if and only if
\begin{equation}\label{eqint2}
\int_{\Delta_{\Phi}(1)}^{+\infty}\frac{dx}{x\Delta_{\Phi}^{-1}(x)}<\infty.
\end{equation}
By a change of variables, letting $Q=\Delta_{\Phi}^{-1}(x)$, (\ref{eqint2}) is equivalent to 
\begin{equation}\label{eqint5}
\int_{1}^{+\infty}\frac{d\Delta_{\Phi}(Q)}{Q\Delta_{\Phi}(Q)}<\infty.
\end{equation}
Now for $t=\Delta_{\Phi}^{-1}(1)$ and $T>t$, an integration by parts gives
\begin{equation}\label{eqint4}
T^{-1}\ln\Delta_{\Phi}(T)+\int_{t}^{T}Q^{-2}\ln(\Delta_{\Phi}(Q))dQ=\int_{t}^T \frac{d\Delta_{\Phi}(Q)}{Q\Delta_{\Phi}(Q)}
\end{equation}
and as $T^{-1}\ln\Delta_{\Phi}(T)>0$, letting $T$ goes to infinity in~\eqref{eqint4}, condition~(\ref{eqint5}) implies
\begin{equation}\label{eqint3}
\int_{1}^{+\infty}Q^{-2}\ln(\Delta_{\Phi}(Q))dQ<\infty.
\end{equation}
Now, using the fact that $\Delta_{\Phi}$ is increasing and assuming that~\eqref{eqint3} holds true, we also have
\begin{equation*}
T^{-1}\ln\Delta_{\Phi}(T)=\int_{T}^{+\infty}Q^{-2}\ln(\Delta_{\Phi}(T))dQ \leq \int_{T}^{+\infty}Q^{-2}\ln(\Delta_{\Phi}(Q))dQ 
\end{equation*}
and therefore letting $T$ goes to infinity in~(\ref{eqint4}), condition~(\ref{eqint3}) implies condition~(\ref{eqint5}). So~(\ref{eqint5}) and~(\ref{eqint3}) are in fact equivalent. Since $\ln\Delta_{\Phi}(Q)=\ln Q+\ln\Phi(Q)$, \eqref{eqint3} is clearly equivalent to
\begin{equation}\label{eqint1}
\int_{1}^{+\infty}Q^{-2}\ln(\Phi(Q))dQ<\infty.
\end{equation}
Hence conditions~\eqref{condA}, \eqref{eqint2}, \eqref{eqint5}, \eqref{eqint3}, \eqref{eqint1} and $\alpha\in \mathcal{B}_d$ are all equivalent, and this ends the proof.
\end{proof}

\section{Technical estimates}\label{app2}

In this second appendix, we derive technical estimates concerning the Lie series method for vector fields. These are well-known (see \cite{Fas90} for instance, where this formalism was first used as far as we aware of), but for completeness we prove the estimates adapted to our need.

\begin{lemma}\label{tech1}
Let $V$ be a bounded real-analytic vector field on $\T^n_s$, $0<\varsigma<s$ and $\tau=\varsigma |V|_{s}^{-1}$. For $t\in \C$ such that $|t|< \tau$, the map $V^t : \T^n_{s-\varsigma} \rightarrow \T^n_s$ is a well-defined real-analytic embedding, and we have 
\[ |V^t-\mathrm{Id}|_{s-\varsigma}\leq |V|_s, \quad |t|<\tau. \]
Moreover, $V^t$ depends analytically on $t$, for $|t|<\tau$.
\end{lemma}

\begin{proof} This is a direct consequence of the existence theorem for analytic differential equations and the analytic dependence on the initial condition: on the domain $\T^n_{s-\varsigma}$, for $|t|<\tau$, $V^t$ is well-defined, depends analytically on $t$ and satisfies the equality
\[ V^t=\mathrm{Id}+\int_{0}^{t}V\circ V^u du. \]
The statement follows.
\end{proof}

\begin{lemma}\label{tech2}
Let $X$ and $V$ be two bounded real-analytic vector fields on $\T^n_s$, and $0<\varsigma<s$. Then
\[ |[X,V]|_{s-\varsigma}\leq 2\varsigma^{-1}|X|_s|V|_s. \]
\end{lemma}

\begin{proof}
First consider a real-analytic function $f$ defined on $\T^n_s$, and let $\mathcal{L}_V f$ the Lie derivative of $f$ along $V$, that is
\[ \mathcal{L}_V f=\frac{d}{dt}(f\circ V^t)_{|t=0}=F'(0), \quad F(t)=f\circ V^t. \]
Now for $|t|<\tau=\varsigma |V|_{s}^{-1}$, by Lemma~\ref{tech1}, $F(t)$ is a well-defined real-analytic function on $\T^n_{s-\sigma}$ and moreover the map $F$ is analytic in $t$, hence by the classical Cauchy estimate 
\[ |\mathcal{L}_V f|_{s-\sigma}=|F'(0)|\leq \tau^{-1}\sup_{|t|<\tau}|f\circ V^t|_{s-\sigma}\leq  \tau^{-1} |f|_s =\varsigma^{-1} |V|_{s}|f|_s. \]
Similarly, we have
\[ |\mathcal{L}_X f|_{s-\sigma} \leq \varsigma^{-1} |X|_{s}|f|_s.  \]
Now, for $1\leq i \leq n$, if $X^i$ and $V^i$ are the components of $X$ and $V$, then $\mathcal{L}_X V_i-\mathcal{L}_V X_i$ are the components of  $[X,V]$, so each component of $[X,V]$ is bounded, on the domain $\T^n_{s-\sigma}$, by $2\varsigma^{-1}|X|_s|V|_s$ and therefore
\[ |[X,V]|_{s-\varsigma}\leq 2\varsigma^{-1}|X|_s|V|_s \]
which is the desired estimate.
\end{proof}

\begin{lemma}\label{tech3}
Let $X$ and $V$ be two bounded real-analytic vector fields on $\T^n_s$, and $0<\varsigma<s$. Assume that $|V|_s \leq (4e)^{-1}\varsigma$. Then for all $|t|\leq 1$, 
\[ |(V^t)^*X|_{s-\sigma}\leq 2|X|_s. \]
\end{lemma}

\begin{proof}
From the general identity
\[ \frac{d}{dt}(V^t)^*X=(V^t)^*[X,V] \]
we have the formal Lie series expansion
\[ (V^t)^*X=\sum_{n\in \N}(n!)^{-1}X_nt^n, \quad X_0=X, \quad X_{n}=[X_{n-1},V], \; n\geq 1. \]
Let $s_i=s-in^{-1}\varsigma$ for $1\leq i \leq n$, so that $s_0=s$ and $s_n=s-\varsigma$. Using Lemma~\ref{tech2}, we have
\[ |X_n|_{s-\varsigma}=|X_n|_{s_n}=|[X_{n-1},V]|_{s_n}\leq 2n\varsigma^{-1}|X_{n-1}|_{s_{n-1}}|V|_{s_{n-1}} \]
and by induction
\[ |X_n|_{s-\varsigma}\leq (2n\varsigma^{-1})^n|X_0|_{s_0}|V|_{s_0}^n=(2n\varsigma^{-1})^n|X|_{s}|V|_{s}^n. \]
By assumption, $|V|_s \leq (4e)^{-1}\varsigma$ so
\[ |X_n|_{s-\varsigma}\leq (2e)^{-n}n^n |X|_s \]
and therefore
\[ |(V^t)^*X|_{s-\sigma}\leq |X|_s \sum_{n\geq 0}(n!)^{-1}(2e)^{-n}(n|t|)^n. \]
Now for any $n\in\N^*$, $n!\geq n^ne^{-n}$, hence $(n!)^{-1}(2e)^{-n}(n|t|)^n \leq (2^{-1}|t|)^n$ and since $|t|\leq 1$, the above series is bounded by $2$. This proves the lemma.
\end{proof}

The above estimate can be easily improved, but for $t\neq 0$, the constant $2$ cannot be replaced by $1$.

 \newcommand{\url}{\texttt}
\addcontentsline{toc}{section}{References}
\bibliographystyle{amsalpha}
\bibliography{AS6}

\end{document}